\documentclass{amsart}
\usepackage{amssymb,latexsym}
\usepackage{amscd,amsthm}

\usepackage[all]{xy}
\usepackage{tikz}

\usepackage{tikz-cd}

\newtheorem{theorem}{Theorem}[section]
\newtheorem{them}{Theorem}

\newtheorem{lemma}[theorem]{Lemma}
\newtheorem{proposition}[theorem]{Proposition}
\newtheorem{corollary}[theorem]{Corollary}

\theoremstyle{definition}
\newtheorem{definition}[theorem]{Definition}

\newtheorem*{remark}{Remark}

\newtheorem{example}[theorem]{Example}

\newtheorem*{setup}{Set up}

\DeclareMathOperator{\Ext}{Ext}
\DeclareMathOperator{\Hom}{Hom}

\DeclareMathOperator{\im}{Im}

\newcommand{\cat}[1]{\mathcal{#1}}           %% font for categories

\newcommand{\class}[1]{\mathcal{#1}}   %% font for classes

\newcommand{\ch}{\textnormal{Ch}(R)}
\newcommand{\cha}[1]{\textnormal{Ch}(\mathcal{#1})}

\newcommand{\tilclass}[1]{\widetilde{\class{#1}}}
\newcommand{\dgclass}[1]{dg\widetilde{\class{#1}}}
\newcommand{\dwclass}[1]{dw\widetilde{\class{#1}}}
\newcommand{\exclass}[1]{ex\widetilde{\class{#1}}}

\newcommand{\rightperp}[1]{#1^{\perp}}
\newcommand{\leftperp}[1]{{}^\perp #1}

\newcommand{\homcomplex}{\mathit{Hom}}

\begin{document}

\title[Homotopy categories of injectives and Gorenstein injectives]{Models for homotopy categories of injectives and Gorenstein injectives}

\author{James Gillespie}
\address{Ramapo College of New Jersey \\
         School of Theoretical and Applied Science \\
         505 Ramapo Valley Road \\
         Mahwah, NJ 07430}
\email[Jim Gillespie]{jgillesp@ramapo.edu}
\urladdr{http://pages.ramapo.edu/~jgillesp/}

\date{\today}

\begin{abstract}%% what do you actually prove?%%
A natural generalization of locally noetherian and locally coherent categories leads us to define locally type $FP_{\infty}$ categories. They include not just all categories of modules over a ring, but also the category of sheaves over any concentrated scheme. In this setting we generalize and study the absolutely clean objects recently introduced in~\cite{bravo-gillespie-hovey}. We show that $\class{D}(\class{AC})$, the derived category of absolutely clean objects, is always compactly generated and that it is embedded in $K(Inj)$, the chain homotopy category of injectives, as a full subcategory containing the DG-injectives. Assuming the ground category $\cat{G}$ has a set of generators satisfying a certain vanishing property, we also show that there is a recollement relating $\class{D}(\class{AC})$ to the (also compactly generated) derived category $\class{D}(\cat{G})$. Finally, we generalize the Gorenstein AC-injectives of~\cite{bravo-gillespie-hovey}, showing that they are the fibrant objects of a cofibrantly generated model structure on $\cat{G}$.
\end{abstract}

\maketitle

\section{Introduction}
Let $\cat{G}$ be a Grothendieck category.  Recall that this is a cocomplete abelian category with a set of generators and such that direct limits are exact.
We say an object $F \in \cat{G}$ is of \emph{type $FP_{\infty}$} if  $\Ext^n_{\cat{G}}(F,-)$ preserves direct limits for all $n \geq 0$. Such objects are automatically finitely presented ($n = 0$). Thinking of these objects as our ``finite'' objects, we call $\cat{G}$ a \emph{locally type $FP_{\infty}$} category if it possesses a generating set $\{G_i\}$ with each $G_i$ of type $FP_{\infty}$. Besides including all Grothendieck categories with a set of finitely generated projective generators, this class of categories includes all locally noetherian and locally coherent categories as well as a vast collection of sheaf and quasi-coherent sheaf categories. Following~\cite{bravo-gillespie-hovey}, we say an object $A$ is \emph{absolutely clean} if $\Ext^1_{\cat{G}}(F,A) = 0$ for all objects $F$ of type $FP_{\infty}$. When $\cat{G}$ is locally noetherian, the absolutely clean objects are precisely the injective objects. When $\cat{G}$ is locally coherent, they are precisely the absolutely pure objects (also called FP-injectives).  For a general locally type $FP_{\infty}$ category, the absolutely clean objects enjoy the same nice properties that injective objects have when $\cat{G}$ is locally noetherian. See Propositions~\ref{prop-abs-clean-properties}/\ref{prop-abs-clean-deconstructible} and Theorems~\ref{them-locally noetherian characterizations}/\ref{them-locally coherent characterizations}.

Now if $\cat{G}$ is locally type $FP_{\infty}$, then so will be the chain complex category $\cha{G}$. Letting $\class{AC}$ denote the class of absolutely clean objects, it inherits the structure of an exact category where the short exact sequences are the usual ones but with all three terms in $\class{AC}$. With respect to this exact structure, an acyclic complex is a complex $A$ which is exact (acyclic) in the usual sense but with each cycle $Z_nA \in \class{AC}$. In the case of modules over a ring $R$, it was explicitly shown in~\cite[Proposition~2.6]{bravo-gillespie} that these are precisely the absolutely clean objects in the category $\ch$. Therefore we will call these \emph{absolutely clean complexes} and denote the class of all of them by $\tilclass{AC}$. Finally, we will say that a complex of injectives $I$ is \emph{AC-injective} if each chain map $A \xrightarrow{} I$ is null homotopic whenever $A$ is an absolutely clean complex. Note that every DG-injective complex is AC-injective since the definition of DG-injective is exactly the same but requires null homotopy when mapping into $I$ from \emph{any} exact complex $A$. Referring to~\cite{neeman-exact category} and~\cite{keller-derived cats} for the notion of the derived category of an exact category, we prove the following in Corollary~\ref{cor-compact generation of abs clean derived cat} and Theorem~\ref{them-injective model for abs clean}.

\begin{them}\label{them-one}
Let $\cat{G}$ be any locally type $FP_{\infty}$ category. Then $\class{D}(\class{AC})$, the derived category of absolutely clean objects, is a compactly generated triangulated category. Moreover, it is equivalent to the chain homotopy category $K(\textnormal{AC-}Inj)$, of all AC-injective complexes.
\end{them}

The above is a generalization of a recent result of Stovicek~\cite{stovicek-purity}. He shows this in the locally coherent case, and in fact he is able to show in this case that $K(\textnormal{AC-}Inj)$ is nothing more that $K(Inj)$, the chain homotopy category of all complexes of injectives. For non-coherent situations, it is not clear if or when all maps $A \xrightarrow{} I$ are null homotopic whenever $A \in \tilclass{AC}$ and $I$ is a complex of injectives.

Now since reading the paper~\cite{becker}, the author has been interested in the relationship between cotorsion pairs and recollement of triangulated categories. (See Section~\ref{subsec-abelian model cats} and Definition~\ref{def-recollement}.) Using the methods of~\cite{gillespie-recollements} we construct several cotorsion pairs that are interrelated in such a way to at once yield a recollement. To state the results, we say an object $A \in \cat{G}$ has \emph{finite projective dimension} if for each $B \in \cat{G}$, there is an $n$ such that $\Ext^m_{\cat{G}}(A,B) = 0$ for all $m \geq n$.

\begin{them}\label{them-two}
Let $\cat{G}$ be a Grothendieck category possessing a set of generators $\{G_i\}$ with each $G_i$ of finite projective dimension. We call such a $\cat{G}$ locally finite dimensional (See Section~\ref{subsec-generators of finite projective dimension}).
\begin{enumerate}
\item Letting $S(Inj)$ denote the chain homotopy category of all exact complexes of injectives, there is a recollement
\[
\xy
(-28,0)*+{S(Inj)};
(0,0)*+{K(Inj)};
(25,0)*+{\class{D}(\cat{G})};
{(-19,0) \ar (-10,0)};
{(-10,0) \ar@<0.5em> (-19,0)};
{(-10,0) \ar@<-0.5em> (-19,0)};
{(10,0) \ar (19,0)};
{(19,0) \ar@<0.5em> (10,0)};
{(19,0) \ar@<-0.5em> (10,0)};
\endxy
.\]

\item Suppose $\{G_i\}$ is contained in the class of all $FP_{\infty}$ objects. We then call $\cat{G}$ a locally finite dimensionally type $FP_{\infty}$ category. Then we have the recollement below and in fact all three triangulated categories are compactly generated:
\[
\xy
(-28,0)*+{S(\class{AC})};
(0,0)*+{\class{D}(\class{AC})};
(25,0)*+{\class{D}(\cat{G})};
{(-19,0) \ar (-10,0)};
{(-10,0) \ar@<0.5em> (-19,0)};
{(-10,0) \ar@<-0.5em> (-19,0)};
{(10,0) \ar (19,0)};
{(19,0) \ar@<0.5em> (10,0)};
{(19,0) \ar@<-0.5em> (10,0)};
\endxy
.\]
Here, $S(\class{AC})$ is the full subcategory of $\class{D}(\class{AC})$ consisting of all exact complexes of absolutely clean objects.

\item With the same hypotheses as the above (2), there is an injective model for $S(\class{AC})$ showing that it is equivalent to $S(\textnormal{AC-}Inj)$. This is the full subcategory of  $K(\textnormal{AC-}Inj)$ consisting of all exact AC-injective complexes. Using these injective models the above recollement becomes
\[
\xy
(-30,0)*+{S(\textnormal{AC-}Inj)};
(0,0)*+{\ K(\textnormal{AC-}Inj)};
(26,0)*+{\ \class{D}(\cat{G})};
{(-19,0) \ar (-10,0)};
{(-10,0) \ar@<0.5em> (-19,0)};
{(-10,0) \ar@<-0.5em> (-19,0)};
{(12,0) \ar (21,0)};
{(21,0) \ar@<0.5em> (12,0)};
{(21,0) \ar@<-0.5em> (12,0)};
\endxy
.\]
\end{enumerate}
\end{them}

\begin{proof}
We emphasize that (2) and (3) have been established in the coherent case by Stovicek in~\cite{stovicek-purity}. Our general versions are the subject of Section~\ref{sec-recollements}. In particular, the above three results are Corollaries~\ref{cor-recollement of krause} and~\ref{cor-FP-infinity recollement of krause} and~\ref{cor-AC-injective recollement of Krause}.
\end{proof}

The remainder of the paper, Sections~\ref{sec-AC-acyclic models} and~\ref{sec-Gorenstein-inj}, is dedicated to Gorenstein homological algebra. Here we again work in the general setting of the locally finite dimensionally type $FP_{\infty}$ categories of Section~\ref{subsec-generators of finite projective dimension}. Following~\cite{bravo-gillespie-hovey}, we define an object $M$ in such a category to be \emph{Gorenstein AC-injective} if
$M=Z_{0}I$ for some exact complex $I$ of injectives for which $\Hom_{\cat{G}}(A,I)$ remains exact for every absolutely clean object $A$.  Letting $\class{GI}$ denote the class of all Gorenstein AC-injectives,
we establish the following result in Theorem~\ref{thm-Gor-module} and Proposition~\ref{prop-cogenerated}.

\begin{them}\label{them-three}
Let $\cat{G}$ be a locally finite dimensionally type $FP_{\infty}$ category. Then
there is a cofibrantly generated abelian model structure on $\cat{G}$ in which each object is cofibrant and the fibrant objects are precisely the Gorenstein AC-injectives. We call this the \textbf{Gorenstein AC-injective model structure}.
\end{them}

The Gorenstein AC-injective model structure allows us to define the stable category of $\cat{G}$, denoted St($\cat{G}$), as the associated homotopy category. It is equivalent to the category of all Gorenstein AC-injectives, modulo $\sim$, where $f \sim g$ if and only if $g-f$ factors through an injective object.

\

\noindent \textbf{Acknowledgments:}
I owe a debt of gratitude to my coauthors of~\cite{bravo-gillespie-hovey}, Daniel Bravo and Mark Hovey. The original idea of this paper was simply to generalize some of the results of that paper to a more general setting.
To this end, the paper started as notes in the Summer of 2014 while preparing the talk~\cite{gillespie-ASTA talk} for the ASTA conference in Spineto, Italy. I would like to thank the conference organizers for the invitation to speak. Finally, I thank Jan Stovicek for the paper~\cite{stovicek-purity}. The final outcome of this paper was highly influenced by the remarkable results appearing in that paper.

\section{Preliminaries}\label{Sec-preliminaries}

The categorical setting for this paper is that of Grothendieck categories,  and we will heavily use Hovey's theory of abelian model categories~\cite{hovey}. We collect some basic information in this section.

\subsection{Grothendieck categories}\label{subsec-Grothendieck cats}

Recall that a Grothendieck category is a cocomplete abelian category $\cat{G}$, with a generating set, and with exact direct limits. We will often refer to~\cite[Chapter~V]{stenstrom}. To orient the reader, we now summarize some standard facts. First, a Grothendieck category is always complete and every object $B \in \cat{G}$ has an injective envelope $E(B)$. In particular, $\cat{G}$ has enough injectives and these can be used to compute $\Ext^n_{\cat{G}}$. A useful fact is that any Grothendieck category is \emph{well-powered}, meaning the class of subobjects of any given object is in fact a set. See~\cite[Prop~IV.6.6]{stenstrom}, although he uses the term \emph{locally small} instead of well-powered. Finally, given any regular cardinal $\gamma$, by~\cite[Corollary~1.69]{adamek-rosicky}, the class  of all $\gamma$-presented objects is essentially small. This means there exists a set of isomorphism representatives for this class.

\subsection{Thick, abelian, and Serre subcategories}\label{subsec-thick-abelian-serre-subcategories}

Let $\cat{S}$ be a non-empty class of objects, or equivalently, a full subcategory of a Grothendieck category $\cat{G}$. There is a hierarchy of nice properties that $\cat{S}$ can have. Consider a short exact sequence
\[ \tag{\text{$*$}}
0 \xrightarrow{} A \xrightarrow{} B \xrightarrow{} C \xrightarrow{} 0. \]
\begin{definition}\label{def-thick abelian-serre}
Given such a class $\class{S} \subseteq \cat{G}$ we say:
\begin{enumerate}
\item $\class{S}$ is a \textbf{thick subcategory} if it is closed under retracts, that is, direct summands, and whenever two out of three of the terms $A,B,C$ in $(*)$ are in $\class{S}$, then so is the third.

\item $\class{S}$ is a \textbf{wide subcategory} if it is an abelian subcategory and closed under extensions. That is, (i), for each $f$ between objects of $\class{S}$, the $\cat{G}$-kernel and $\cat{G}$-cokernel are back in $\class{S}$. And (ii), in $(*)$ we have $A,C \in \class{S}$ implies $B \in \class{S}$.

\item $\class{S}$ is a \textbf{Serre subcategory} if $B \in \class{S}$ if and only if $A,C \in \class{S}$.
\end{enumerate}
\end{definition}

Note that each type of subcategory must contain 0 and be \emph{replete} meaning it is closed under isomorphic objects. An easy argument shows that a wide subcategory must be closed under retracts. So one easily proves the following lemma.

\begin{lemma}\label{lemma-thick-wide-serre-subcats}
Any Serre subcategory is wide and any wide subcategory is thick.
\end{lemma}

\subsection{Cotorsion pairs and abelian model structures}\label{subsec-abelian model cats}

Let $\cat{A}$ be a bicomplete abelian category. Hovey showed in~\cite{hovey} that an abelian model structure on $\cat{A}$ is nothing more than two nicely related cotorsion pairs in $\cat{A}$.  By definition, a pair of classes $(\class{X},\class{Y})$ in $\cat{A}$ is called a \emph{cotorsion pair} if $\class{Y} = \rightperp{\class{X}}$ and $\class{X} = \leftperp{\class{Y}}$. Here, given a class of objects $\class{C}$ in $\cat{A}$, the right orthogonal  $\rightperp{\class{C}}$ is defined to be the class of all objects $X$ such that $\Ext^1_{\cat{A}}(C,X) = 0$ for all $C \in \class{C}$. Similarly, we define the left orthogonal $\leftperp{\class{C}}$. We call the cotorsion pair \emph{hereditary} if $\Ext^i_{\cat{A}}(X,Y) = 0$ for all $X \in \class{X}$, $Y \in \class{Y}$, and $i \geq 1$. The cotorsion pair is \emph{complete} if it has enough injectives and enough projectives. This means that for each $A \in \cat{A}$ there exist short exact sequences $0 \xrightarrow{} A \xrightarrow{} Y \xrightarrow{} X \xrightarrow{} 0$ and $0 \xrightarrow{} Y' \xrightarrow{} X' \xrightarrow{} A \xrightarrow{} 0$ with $X,X' \in \class{X}$ and $Y,Y' \in \class{Y}$.
Besides their connection to abelian model structures which we describe next, cotorsion pairs are fundamental in modern homological algebra. There are several good references. In particular we will refer to~\cite{enochs-jenda-book} and~\cite{hovey}.

The main theorem of~\cite{hovey} showed that an abelian model structure on $\cat{A}$ is equivalent to a triple $(\class{Q},\class{W},\class{R})$ of classes of objects in $\cat{A}$ for which $\class{W}$ is thick and $(\class{Q} \cap \class{W},\class{R})$ and $(\class{Q},\class{W} \cap \class{R})$ are each complete cotorsion pairs. By \emph{thick} we mean that the class $\class{W}$ is closed under retracts (i.e., direct summands) and satisfies that whenever two out of three terms in a short exact sequence are in $\class{W}$, then so is the third. In this case, $\class{Q}$ is precisely the class of cofibrant objects of the model structure, $\class{R}$ are precisely the fibrant objects, and $\class{W}$ is the class of trivial objects. We therefore denote an abelian model structure $\class{M}$ as a triple $\class{M} = (\class{Q},\class{W},\class{R})$ and for short
we will denote the two associated cotorsion pairs above by $(\tilclass{Q},\class{R})$ and $(\class{Q},\tilclass{R})$. We say that $\class{M}$ is \emph{hereditary} if both of these associated cotorsion pairs are hereditary. We will also call any abelian model structure $\class{M} = (\class{Q},\class{W},\class{R})$ a \emph{Hovey triple}.

By the \emph{core} of a cotorsion pair $(\class{X},\class{Y})$ we mean $\class{X} \cap \class{Y}$, and so by the \emph{core} of an abelian model structure $\class{M} = (\class{Q},\class{W},\class{R})$ we mean the class $\class{Q} \cap \class{W} \cap \class{R}$. A recent result appearing in~\cite{gillespie-hovey triples} gives useful criteria to help one find and construct an abelian model structure. It says that whenever $(\tilclass{Q},\class{R})$ and $(\class{Q},\tilclass{R})$ are complete hereditary cotorsion pairs with equal cores and $\tilclass{R} \subseteq \class{R}$, then there is a unique thick class $\class{W}$ yielding a Hovey triple $\class{M} = (\class{Q},\class{W},\class{R})$ with $\class{Q} \cap \class{W} = \tilclass{Q}$ and $\class{W} \cap \class{R} = \tilclass{R}$. Besides~\cite{hovey} we will refer to~\cite{hovey-model-categories} for any other basics from the theory of model categories. The following lemma will turn out to be especially useful.

\begin{lemma}\label{lemma-half related}
Suppose we have two complete cotorsion pairs $(\class{Q},\tilclass{R})$ and  $(\tilclass{Q},\class{R})$ in an abelian category and that we also have a thick class $\class{W}$. Then the following hold.
\begin{enumerate}
\item  If $\tilclass{Q} = \class{Q} \cap \class{W}$ and $\tilclass{R} \subseteq \class{W}$, then also $\tilclass{R} = \class{W} \cap \class{R}$. That is, $(\class{Q},\class{W},\class{R})$ is a Hovey triple.
\item  If $\tilclass{R} = \class{W} \cap \class{R}$ and $\tilclass{Q} \subseteq \class{W}$, then also $\tilclass{Q} = \class{Q} \cap \class{W}$. That is, $(\class{Q},\class{W},\class{R})$ is a Hovey triple.
\end{enumerate}
\end{lemma}

\begin{proof}
The statements are proved similarly, and we will prove (1). We have by assumption that $\tilclass{Q} \subseteq \class{Q}$ and consequently $\tilclass{R} \subseteq \class{R}$. We are also assuming $\tilclass{R} \subseteq \class{W}$, and so we have $\tilclass{R} \subseteq \class{W} \cap \class{R}$.

It is left to show $\tilclass{R} \supseteq \class{W} \cap \class{R}$. Letting $X \in \class{W} \cap \class{R}$, we use completeness of the cotorsion pair $(\class{Q},\tilclass{R})$ to find a short exact sequence $0 \xrightarrow{} X \xrightarrow{} R \xrightarrow{} Q \xrightarrow{} 0$ with $R \in \tilclass{R}$ and $Q \in \class{Q}$. We see that this forces $Q \in \class{Q} \cap \class{W} = \tilclass{Q}$. Hence the sequence must split, forcing $X$ to be a retract of an object in $\tilclass{R}$. So $X$ is also in $\tilclass{R}$.
\end{proof}

\subsection{Recollement situations}
Here we define what is meant by a recollement of triangulated categories. The standard reference is~\cite{BBD-perverse sheaves}, although the definitions below will suffice for our purposes.

\begin{definition}\label{def-localization sequence}
Let $\class{T}' \xrightarrow{F} \class{T} \xrightarrow{G} \class{T}''$ be a sequence of exact functors between triangulated categories. We say it is a \emph{localization sequence} when there exists right adjoints $F_{\rho}$ and $G_{\rho}$ giving a diagram of functors as below with the listed properties.
$$\begin{tikzcd}
\class{T}'
\rar[to-,
to path={
([yshift=0.5ex]\tikztotarget.west) --
([yshift=0.5ex]\tikztostart.east) \tikztonodes}][swap]{F}
\rar[to-,
to path={
([yshift=-0.5ex]\tikztostart.east) --
([yshift=-0.5ex]\tikztotarget.west) \tikztonodes}][swap]{F_{\rho}}
& \class{T}
\rar[to-,
to path={
([yshift=0.5ex]\tikztotarget.west) --
([yshift=0.5ex]\tikztostart.east) \tikztonodes}][swap]{G}
\rar[to-,
to path={
([yshift=-0.5ex]\tikztostart.east) --
([yshift=-0.5ex]\tikztotarget.west) \tikztonodes}][swap]{G_{\rho}}
& \class{T}'' \\
\end{tikzcd}$$
\begin{enumerate}
\item The right adjoint $F_{\rho}$ of $F$ satisfies $F_{\rho} \circ F \cong \text{id}_{\class{T}'}$.
\item The right adjoint $G_{\rho}$ of $G$ satisfies $G \circ G_{\rho} \cong \text{id}_{\class{T}''}$.
\item For any object $X \in \class{T}$, we have $GX = 0$ iff $X \cong FX'$ for some $X' \in \class{T}'$.
\end{enumerate}
The notion of a \emph{colocalization sequence} is the dual. That is, there must exist left adjoints $F_{\lambda}$ and $G_{\lambda}$ with the analogous properties.
\end{definition}

Note the similarity in the definitions above to the notion of a split exact sequence, but for adjunctions. It is true that if  $\class{T}' \xrightarrow{F} \class{T} \xrightarrow{G} \class{T}''$ is a localization sequence then  $\class{T}'' \xrightarrow{G_{\rho}} \class{T} \xrightarrow{F_{\rho}} \class{T}'$ is a colocalization sequence and if  $\class{T}' \xrightarrow{F} \class{T} \xrightarrow{G} \class{T}''$ is a colocalization sequence then  $\class{T}'' \xrightarrow{G_{\lambda}} \class{T} \xrightarrow{F_{\lambda}} \class{T}'$ is a localization sequence. This brings us to the definition of a recollement where the sequence of functors  $\class{T}' \xrightarrow{F} \class{T} \xrightarrow{G} \class{T}''$ is both a localization sequence and a colocalization sequence.

\begin{definition}\label{def-recollement}
Let $\class{T}' \xrightarrow{F} \class{T} \xrightarrow{G} \class{T}''$ be a sequence of exact functors between triangulated categories. We say $\class{T}' \xrightarrow{F} \class{T} \xrightarrow{G} \class{T}''$ induces a \emph{recollement} if it is both a localization sequence and a colocalization sequence as shown in the picture.
\[
\xy
(-20,0)*+{\class{T}'};
(0,0)*+{\class{T}};
(20,0)*+{\class{T}''};
{(-18,0) \ar^{F} (-2,0)};
{(-2,0) \ar@/^1pc/@<0.5em>^{F_{\rho}} (-18,0)};
{(-2,0) \ar@/_1pc/@<-0.5em>_{F_{\lambda}} (-18,0)};
{(2,0) \ar^{G} (18,0)};
{(18,0) \ar@/^1pc/@<0.5em>^{G_{\rho}} (2,0)};
{(18,0) \ar@/_1pc/@<-0.5em>_{G_{\lambda}} (2,0)};
\endxy
\]
\end{definition}
So the idea is that a recollement is a colocalization sequence ``glued'' with a localization sequence.

\section{Locally type $FP_{\infty}$ categories}\label{Sec-locally type FP-infinity}

The point of this section is to introduce the new class of categories, and objects, that we will be working with throughout the rest of the paper. These are the locally type $FP_{\infty}$ categories and the absolutely clean objects.
So let $\cat{G}$ be a Grothendieck category.  Note that for any object $C$, and any direct system $\{X_i\}_{i \in I}$, there is a canonical map $\xi_n : \varinjlim \Ext^n_{\cat{G}}(C,X_i) \xrightarrow{} \Ext^n_{\cat{G}}(C,\varinjlim X_i)$ for each $n \geq 0$. To say $\Ext^n_{\cat{G}}(C,-)$ \emph{preserves direct limits} means that $\xi_n$ is an isomorphism for each direct system.

\begin{definition}\label{def-objects of type FP-infinity}
An object $F \in \cat{G}$ is said to be of \textbf{type $\boldsymbol{FP_{\infty}}$} if the functors $\Ext^n_{\cat{G}}(F,-)$ preserve direct limits for all $n \geq 0$.
\end{definition}

Recall that an object $C \in \cat{G}$ is called \textbf{finitely presented} if $\Hom_{\cat{G}}(C,-)$ preserves direct limits, that is, when $\xi_0$ is an isomorphism for each direct system. Also, $C$ is called \textbf{finitely generated} when $\Hom_{\cat{G}}(C,-)$ preserves direct unions of subobjects of any given object. So any object of type $FP_{\infty}$ is certainly finitely presented and hence finitely generated.

\begin{example}\label{Example-fg-projectives}
Any finitely generated projective object must be of type $FP_{\infty}$. (Reason) $\Ext^n_{\cat{G}}(P,-)$ vanishes for $n > 0$ and projective objects $P$. So it is enough to show that for any finitely generated projective $P$, and direct system $\{X_i\}_{i \in I}$, the canonical map $\xi_0 : \varinjlim \Hom_{\cat{G}}(P,X_i) \xrightarrow{} \Hom_{\cat{G}}(P,\varinjlim X_i)$ is an isomorphism. The fact that $\xi_0$ is a monomorphism follows just because $P$ is finitely generated; see paragraph three of the proof of~\cite[Prop.V.3.4]{stenstrom} for the argument. To see $\xi_0$ is an epimorphism, consider a morphism $\alpha : P \xrightarrow{} \varinjlim X_i$. Let $X'_i$ denote $\im{(X_i \xrightarrow{} \varinjlim X_i)}$, so that $\varinjlim X_i = \Sigma X'_i$. Since this is a direct union of subobjects, and $P$ is finitely generated, $\alpha$ must factor through some $X'_i$ by a morphism $\alpha' : P \xrightarrow{} X'_i$. But $P$ is projective and so this $\alpha'$ lifts over the epimorphism $X_i \twoheadrightarrow X'_i$. This shows $\xi_0$ is an epimorphism.
\end{example}

Recalling the notion of a thick subcategory from Section~\ref{subsec-thick-abelian-serre-subcategories}, we have the following proposition.

\begin{proposition}\label{prop-thickness of FP-infinity}
For any Grothendieck category $\cat{G}$, the class of all objects of type $FP_{\infty}$ is a thick subcategory.
\end{proposition}

\begin{proof}
Consider a short exact sequence $0 \xrightarrow{} A \xrightarrow{} B \xrightarrow{} C \xrightarrow{} 0$ and a direct system $\{X_i\}_{i \in I}$. It gives rise to a long exact sequence of direct systems
$$0 \xrightarrow{} \{\Hom_{\cat{G}}(C, X_i)\}_{i \in I} \xrightarrow{} \{\Hom_{\cat{G}}(B, X_i)\}_{i \in I} \xrightarrow{} \{\Hom_{\cat{G}}(A, X_i)\}_{i \in I} \xrightarrow{}$$
$$\{\Ext^1_{\cat{G}}(C, X_i)\}_{i \in I} \xrightarrow{} \{\Ext^1_{\cat{G}}(B, X_i)\}_{i \in I} \xrightarrow{} \{\Ext^1_{\cat{G}}(A, X_i)\}_{i \in I} \xrightarrow{} \cdots $$ Since direct limits (of abelian groups) are exact, we get a long exact sequence
$$0 \xrightarrow{} \varinjlim \Hom_{\cat{G}}(C, X_i) \xrightarrow{} \varinjlim \Hom_{\cat{G}}(B, X_i) \xrightarrow{} \varinjlim \Hom_{\cat{G}}(A, X_i) \xrightarrow{}$$
$$\varinjlim \Ext^1_{\cat{G}}(C, X_i) \xrightarrow{} \varinjlim \Ext^1_{\cat{G}}(B, X_i) \xrightarrow{} \varinjlim \Ext^1_{\cat{G}}(A, X_i) \xrightarrow{} \cdots $$
The natural maps $\xi_0,\xi_1,\xi_2, \dots $ connect this long exact sequence to the long exact sequence below:
$$0 \xrightarrow{}  \Hom_{\cat{G}}(C, \varinjlim X_i) \xrightarrow{} \Hom_{\cat{G}}(B, \varinjlim X_i) \xrightarrow{} \Hom_{\cat{G}}(A, \varinjlim X_i) \xrightarrow{}$$
$$\Ext^1_{\cat{G}}(C, \varinjlim X_i) \xrightarrow{} \Ext^1_{\cat{G}}(B, \varinjlim X_i) \xrightarrow{} \Ext^1_{\cat{G}}(A, \varinjlim X_i) \xrightarrow{} \cdots $$
Now an application of the five lemma shows that whenever two out of three of $A,B,C$ are of type $FP_{\infty}$, then so is the third.

It is left to show that the class of $FP_{\infty}$ objects is closed under retracts. So say $A \oplus B$ is of type $FP_{\infty}$ and $\{X_i\}_{i \in I}$ is a direct system. Then $\{  \Ext^n_{\cat{G}}(A \oplus B, X_i) \}_{i \in I}$ is a direct system of abelian groups, which we note is isomorphic to a direct system $\{\Ext^n_{\cat{G}}(A, X_i) \oplus \Ext^n_{\cat{G}}(B, X_i)\}_{i \in I}$. By Lemma~\ref{lemma-direct limit} below we have
$$\varinjlim [\Ext^n_{\cat{G}}(A, X_i) \oplus \Ext^n_{\cat{G}}(B, X_i)] \cong [\varinjlim \Ext^n_{\cat{G}}(A, X_i)] \oplus [\varinjlim \Ext^n_{\cat{G}}(B,X_i)].$$
This means we have an isomorphism in the top row of the commutative diagram:
$$\begin{CD}
  \varinjlim \Ext^n_{\cat{G}}(A \oplus B, X_i) @>>> [\varinjlim \Ext^n_{\cat{G}}(A, X_i)] \oplus [\varinjlim \Ext^n_{\cat{G}}(B,X_i)]  \\
@VV \xi_{A \oplus B, n} V  @VV\xi_{A,n} \oplus \xi_{B,n} V \\
\Ext^n_{\cat{G}}(A \oplus B, \varinjlim X_i) @>>> \Ext^n_{\cat{G}}(A, \varinjlim X_i) \oplus \Ext^n_{\cat{G}}(B,\varinjlim X_i)  \\
\end{CD}$$
Clearly the bottom row is also an isomorphism, as is $\xi_{A \oplus B, n}$ by hypothesis. Thus $\xi_{A,n} \oplus \xi_{B,n}$ is an isomorphism.
One can check that this implies the summands $\varinjlim \Ext^n_{\cat{G}}(A,X_i) \xrightarrow{\xi_{A,n}} \Ext^n_{\cat{G}}(A,\varinjlim X_i)$ and $\varinjlim \Ext^n_{\cat{G}}(B,X_i) \xrightarrow{\xi_{B,n}} \Ext^n_{\cat{G}}(B,\varinjlim X_i)$ are all isomorphisms.
\end{proof}

\begin{lemma}\label{lemma-direct limit}
Let $I$ be a directed set, and let $\{A_i,\alpha_{ij}\}$ and $\{B_i,\beta_{ij}\}$ each be direct systems, over $I$, of abelian groups. Then the direct system $\{A_i \oplus B_i, \alpha_{ij} \oplus \beta_{ij}\}$ satisfies
$\varinjlim (A_i \oplus B_i) \cong \varinjlim A_i \oplus \varinjlim B_i$.
\end{lemma}

\begin{proof}
Show that $\varinjlim A_i \oplus \varinjlim B_i$ satisfies the universal property that is unique to $\varinjlim (A_i \oplus B_i)$. We leave the details to the reader.
\end{proof}

But there is no guarantee that a Grothendieck category possesses any nonzero objects of type $FP_{\infty}$. So we propose Definition~\ref{def-locally FP-infinity} below in the spirit of locally finitely generated and locally finitely presented categories. Recall that a Grothendieck category $\cat{G}$ is called \textbf{locally finitely generated} if it has a set of finitely generated generators. This is equivalent to saying that each $C \in \cat{G}$ is a direct union of finitely generated subobjects~\cite[pp.~122]{stenstrom}. $\cat{G}$ is called \textbf{locally finitely presented} if it has a set of finitely presented generators. This is equivalent to saying that each $C \in \cat{G}$ is a direct limit of finitely presented objects~\cite[Theorem~1.11]{adamek-rosicky}.

\begin{definition}\label{def-locally FP-infinity}
We say that $\cat{G}$ \textbf{locally type $\boldsymbol{FP_{\infty}}$} if it has a generating set consisting of objects of type $FP_{\infty}$.
\end{definition}

Note that any locally type $FP_{\infty}$ category is certainly locally finitely presented and hence locally finitely generated. The following lemma will prove to be important. It is based on similar results that can be found in the work of Stovicek. In particular, see~\cite[Proposition B.3]{stovicek-purity}.

\begin{lemma}\label{lemma-stoviceks lemma}
Let $\cat{G}$ be a Grothendieck category. Assume we have a generating set $\class{S}$ consisting of finitely generated objects such that $\class{S}$ is closed under both finite direct sums and taking kernels of epimorphisms between objects of $\class{S}$. Then the following hold.
\begin{enumerate}
\item Every finitely generated object is a quotient of an object in $\class{S}$.
\item For all $F \in \class{S}$ and $A \in \rightperp{\class{S}}$, we have $\Ext^2_{\cat{G}}(F,A) = 0$.
\item $\rightperp{\class{S}}$ is closed under taking cokernels of monomorphisms.
\item $(\leftperp{(\rightperp{\class{S}})},\rightperp{\class{S}})$ is an hereditary cotorsion pair, and it is small in the sense of~\cite[Def.~6.4]{hovey}.  Explicitly, the generating monomorphisms can be taken to be the set $I$ of all monos $F' \hookrightarrow F''$ with $F'$, $F''$, and $F''/F'$  all in $\class{S}$.
\end{enumerate}
\end{lemma}

\begin{proof}
(1) is easy. Indeed if $F$ is finitely generated, then since $\class{S}$ is generating we can find an epimorphism $\oplus_{i \in I} F_i \twoheadrightarrow F$ where each $F_i \in \class{S}$. Write $F = \Sigma_{i \in I} F'_i$ where $F'_i = \im{(F_i \hookrightarrow \oplus F_i \twoheadrightarrow F)}$. Since $F$ is finitely generated, there exists a finite subset $J \subseteq I$ such that $F = \Sigma_{i \in J} F'_i$. This means $\oplus_{i \in J} F_i \twoheadrightarrow F$ is still an epimorphism. By hypothesis, $\oplus_{i \in J} F_i \in \class{S}$.

For (2), let $F \in \class{S}$ and $A \in \rightperp{\class{S}}$. Recall the Yoneda description of the group $\Ext^2_{\cat{G}}(F,A)$. Its elements are equivalence classes of exact sequences of the form
$$\epsilon: \ \  0 \xrightarrow{} A \xrightarrow{} X_2  \xrightarrow{} X_1 \xrightarrow{} F \xrightarrow{} 0$$ As described in~\cite[pp.~79]{weibel}, the equivalence relation is \emph{generated} by the relation $\sim$, where $\epsilon' \sim \epsilon$ means there exists some commutative diagram of the form
$$\begin{CD}
\epsilon' : \ \ 0 @>>>  A @>>> X'_2 @>>> X'_1 @>>> F  @>>> 0 \\
@. @|  @VVV @VVV @|  @.\\
\epsilon : \ \ 0 @>>> A @>>>  X_2  @>>> X_1 @>>> F  @>>> 0 \\
\end{CD}$$
Now let $\epsilon \in \Ext^2_{\cat{G}}(F,A)$ be arbitrary. Our goal is to show that it is equivalent to the split 2-sequence $\sigma : \ \ 0 \xrightarrow{} A \xrightarrow{} A \oplus X  \xrightarrow{} X \oplus F \xrightarrow{} F \xrightarrow{} 0$.  We write
$\epsilon: \ \  0 \xrightarrow{} A \xrightarrow{k} X_2  \xrightarrow{f} X_1 \xrightarrow{c} F \xrightarrow{} 0$, and note that, by~\cite[Lemma~V.3.3]{stenstrom}, we can find a finitely generated subobject $S \subseteq X_1$ such that $c(S) = F$. By (1), there is an epimorphism $F' \xrightarrow{p} S$ where $F' \in \class{S}$. Letting $X'_2$ denote the pullback of $X_2 \xrightarrow{f} X_1 \xleftarrow{p} F'$, one constructs a morphism of exact 2-sequences:
$$\begin{CD}
\epsilon' : \ \ 0 @>>>  A @>k'>> X'_2 @>f'>> F' @>cp>> F  @>>> 0 \\
@. @|  @VVp' V @VVpV @|  @.\\
\epsilon : \ \ 0 @>>> A @>k>>  X_2  @>f>> X_1 @>c>> F  @>>> 0 \\
\end{CD}$$
Looking at the short exact sequence $0 \xrightarrow{} A \xrightarrow{k'} X'_2  \xrightarrow{f'} \im{f'} \xrightarrow{} 0$, we note that $\im{f'} = \ker{(cp)} \in \class{S}$, by hypothesis. Since $A \in \rightperp{\class{S}}$, this means the short exact sequence splits. Using this fact, one can now easily construct a morphism of exact 2-sequences, showing $\epsilon' \sim \sigma$. This means that $\sigma$, $\epsilon'$, and $\epsilon$ all represent the same element, namely zero, in the Yoneda description of $\Ext^2_{\cat{G}}(F,A)$.

(3) follows easily from (2). Indeed let $0 \xrightarrow{} A \xrightarrow{} A' \xrightarrow{}  C  \xrightarrow{} 0$ be a short exact sequence with $A,A' \in \rightperp{\class{S}}$. For $F \in \class{S}$, apply $\Hom_{\cat{G}}(F,-)$, and the corresponding long exact sequence shows $\Ext^1_{\cat{G}}(F,C) = 0$.

We now focus on (4). It is clear that $\class{S}$ cogenerates a cotorsion pair $(\leftperp{(\rightperp{\class{S}})},\rightperp{\class{S}})$. We refer to~\cite[Definition~6.4]{hovey} for the definition of a small cotorsion pair. Considering the hypotheses on $\class{S}$, and property (1) above, it is enough to show that if an object $C \in \cat{G}$ is injective with respect to the set of all monomorphisms $F' \hookrightarrow F''$ with $F'$ , $F''$, and $F''/F'$  each in $\class{S}$, then $C \in \rightperp{\class{S}}$.
So let $C \in \class{G}$ have this extension property and let $0 \xrightarrow{} C \xrightarrow{} X \xrightarrow{} F \xrightarrow{} 0$ be any short exact sequence with $F \in \class{S}$. The proof will be complete once we show this sequence splits. But using a variation on the argument proving the above part~(2), we can construct a morphism of short exact sequences with $F' F''  \in \class{S}$.
$$\begin{CD}
0 @>>> F' @>>> F'' @>>> F  @>>> 0 \\
@. @VVV @VVV @| @. \\
0 @>>> C  @>>> X @>>> F  @>>> 0. \\
\end{CD}$$
The assumption on $C$ means there is a morphism $F'' \xrightarrow{} C$ producing a commutative triangle in the upper left corner. This is in fact equivalent, by a fact sometimes called ``the homotopy lemma'', to a map $F \xrightarrow{} X$ producing a commutative triangle in the lower right corner.  This is a splitting.
\end{proof}

With Lemma~\ref{lemma-stoviceks lemma}, and the appropriate setting of a locally type $FP_{\infty}$ category, we may now go on to define absolutely clean objects. We are following~\cite{bravo-gillespie-hovey}.

\begin{definition}\label{def-abs-clean}
Let $\cat{G}$ be a locally type $FP_{\infty}$ category. Let $A \in \cat{G}$ be an object of $\cat{G}$ and
let $\epsilon : 0 \xrightarrow{} X \xrightarrow{} Y \xrightarrow{} Z \xrightarrow{}  0$ be a short exact sequence.
\begin{itemize}
\item We say $\epsilon$ is \textbf{clean} if $\Hom_{\cat{G}}(F, \epsilon)$ remains exact for each $F$ of type $FP_{\infty}$.

\item We say that $A$ is \textbf{absolutely clean} if $\Ext^1_{\cat{G}}(F,A) = 0$ for all $F$ of type $FP_{\infty}$. We denote the class of all absolutely clean objects by $\class{AC}$.
\end{itemize}
\end{definition}

Note that $A$ is absolutely clean if and only if each short exact sequence starting with $A$ is clean. We will now see that the class of absolutely clean objects possesses many nice properties. First, observe that if $\cat{G}$ is locally type $FP_{\infty}$, then the hypotheses of Lemma~\ref{lemma-stoviceks lemma} are satisfied when we take $\class{S}$ to be a set of isomorphism representatives for the class of all objects of type $FP_{\infty}$. Denoting this choice of the set $\class{S}$ by $FP_{\infty}(\cat{G})$, the following corollary is immediate.

\begin{proposition}\label{prop-abs-clean-cotorsion-pair}
Let $\cat{G}$ be a locally type $FP_{\infty}$ category. Then $(\leftperp{\class{AC}},\class{AC})$ is an hereditary cotorsion pair. It is small, and hence (functorially) complete. Explicitly, we can take the set $I$ of generating monomorphisms to be the set of all monomorphisms $F' \hookrightarrow F''$ with $F'$, $F''$, and $F''/F'$  all in $FP_{\infty}(\cat{G})$.
\end{proposition}

The next proposition is the analog of~\cite[Prop.~2.5]{bravo-gillespie-hovey}.

\begin{proposition}\label{prop-abs-clean-properties}
Let $\cat{G}$ be a locally type $FP_{\infty}$ category. Then the class $\class{AC}$ of absolutely clean objects satisfies the following:
\begin{enumerate}
\item The class $\class{AC}$ of absolutely clean objects is coresolving; that is, it contains the injectives and is closed under extensions and cokernels of monomorphisms.

\item If $A$ is absolutely clean, then $\Ext^n_{\cat{G}}(F,A) = 0$ for all $n > 0$ and $F$ of type $FP_{\infty}$.

\item The class $\class{AC}$ of absolutely clean objects is closed under pure subobjects and pure quotients. In fact, it is closed under clean subobjects and clean quotients.

\item The class $\class{AC}$ of absolutely clean objects is closed under direct products, direct sums, retracts, direct limits, and transfinite extensions.
\end{enumerate}
\end{proposition}

\begin{proof}
(1) is clear, and (2) follows from the fact that it holds for $n = 2$ along with a ``dimension shifting'' argument.

For (3), let $0 \xrightarrow{} C \xrightarrow{} A \xrightarrow{} A/C \xrightarrow{}  0$ be a clean exact sequence, with $A \in \class{AC}$. Then it is easy to argue that $C \in \class{AC}$, and hence $A/C \in \class{AC}$ too.

For (4), we see that $\class{AC}$ is closed under extensions, direct products, and retracts, since it is the right hand side of a cotorsion pair. $\class{AC}$ is closed under direct limits since $\Ext^1_{\cat{G}}(F,-)$
commutes with direct limits whenever $F$ is type $FP_{\infty}$. Now since $\class{AC}$ is closed under both extensions and direct limits, it is therefore closed under transfinite extensions. Finally, $\class{AC}$ is closed under direct sums, since any direct sum can be realized as the direct limit of all the finite sums of the direct summands.
\end{proof}

\begin{proposition}\label{prop-abs-clean-deconstructible}
Let $\cat{G}$ be a locally type $FP_{\infty}$ category. Then there is a set $\class{S}$ of absolutely clean objects such that every absolutely clean object is a transfinite extension of $\class{S}$. Equivalently, there exists a regular cardinal $\gamma$ such that each absolutely clean object is a transfinite extension of $\gamma$-presented absolutely clean objects.
\end{proposition}

\begin{proof}
Since $\cat{G}$ is a locally type $FP_{\infty}$ category, it is in particular locally finitely presentable. So
an application of~\cite[Theorem~2.33]{adamek-rosicky} immediately leads to the following:

\noindent \underline{Fact}: There exists a regular cardinal $\gamma$ such that for any object $C \in \cat{G}$, there exists a nonzero pure subobject $P \subseteq C$ such that $P$ is $\gamma$-presented.

Moreover, by~\cite[Corollary~1.69]{adamek-rosicky}, the class  of all $\gamma$-presented objects is essentially small. Let $\textbf{Pres}_{\gamma}\cat{G}$ denote a set of isomorphism representatives for this class.  Finally, set $\class{S} = \textbf{Pres}_{\gamma}\cat{G} \cap \class{AC}$. Now the argument from~\cite[Prop.~2.6 ]{bravo-gillespie-hovey} shows that every absolutely clean object is a transfinite extension of $\class{S}$. Although that proposition is proved for modules over a ring, the above observations make it clear that the argument carries to our setting. The properties of purity used in the argument can be found in~\cite[Appendix~A]{gillespie-G-derived}, stated in the current generality.
\end{proof}

\subsection{Examples of locally type $\boldsymbol{FP_{\infty}}$ categories}\label{subsec-Examples of FP-infinity categories}
The remainder of this section turns to provide examples of Grothendieck categories of type $FP_{\infty}$.

\begin{example}\label{example-modules}
For any ring $R$, the category $R$-Mod of (left) $R$-modules is locally type $FP_{\infty}$. Indeed, ${}_RR$ is a generator of type $FP_{\infty}$. In fact, Example~\ref{Example-fg-projectives} provides an easy generalization of this: Any Grothendieck category possessing a generating set consisting of finitely generated projective objects must be locally type $FP_{\infty}$.
\end{example}

\begin{example}\label{example-chain complexes}
It is well known that if $\cat{G}$ is a Grothendieck category, then so is the chain complex category $\cha{G}$. If $\{G_i\}$ is a generating set for $\cat{G}$, then $\{D^n(G_i)\}$ is a generating set for $\cha{G}$. The notation $D^n(A)$ denotes the complex that is $A$ in degrees $n$ and $n-1$, connected by $1_A$, and 0 elsewhere. One can show that there are natural isomorphisms $\Ext^m_{\cha{G}}(D^n(A),Y) \cong \Ext^m_{\cat{G}}(A,Y_n)$ for all $m \geq 0$. It follows that if $\cat{G}$ is locally type $FP_{\infty}$, then so is $\cha{G}$.
\end{example}

\begin{example}\label{example-sheaves}
Locally type $FP_{\infty}$ categories naturally arise in algebraic geometry as categories of sheaves.
We will explain this briefly and refer to Murfet's notes~\cite{murfet-notes} for further details.

Following~\cite[CON: Def.~3]{murfet-notes}, a scheme $X$ is called \emph{concentrated} if it is quasi-compact and quasi-separated. By~\cite[COS: Lemma~8]{murfet-notes} we see that a scheme $X$ is concentrated if and only if its underlying topological space is \emph{quasi-noetherian}. So the following proposition gives a nice class of Grothendieck categories of type $FP_{\infty}$.

\begin{proposition}
Let $(X,\class{O}_X)$ be a quasi-noetherian ringed space. For example, let $X$ be a concentrated scheme. Then the category $\class{O}_X$-Mod of all sheaves of $\class{O}_X$-modules is a locally type $FP_{\infty}$ category.
\end{proposition}

\begin{proof}
Recall that there is a standard set of generators for $\class{O}_X$-Mod. It is the set of $\class{O}_X$-modules $\{j!(\class{O}_U)\}$ where $U \subseteq X$ ranges through the open subsets of $X$. In fact it is enough to let $U \subseteq X$ range through just a basis for $X$~\cite[MRS: Cor.~31]{murfet-notes}. But the definition of a quasi-noetherian space $X$ implies that the quasi-compact open subsets of $X$ form a basis for $X$.
So we will be done once we show that $j!(\class{O}_U)$ is of type $FP_{\infty}$ whenever $U \subseteq X$ is a quasi-compact open subset. Now there is an isomorphism $$\Ext^n(j!(\class{O}_U), -) \cong  H^n(U,(-)|_U).$$ Moreover, a quasi-compact open subset $U \subseteq X$ is itself quasi-noetherian~\cite[see COS: Def.~4]{murfet-notes}. So finally, combining~\cite[see COS: proof of Prop.~23]{murfet-notes} and~\cite[see COS: Them~26]{murfet-notes} we conclude that $\Ext^n(j!(\class{O}_U), -)$ preserves direct limits.
\end{proof}
\end{example}

\subsection{Locally noetherian categories}\label{subsec-locally noetherian categories}
Here we show in detail how locally noetherian categories are particularly nice locally type $FP_{\infty}$ categories. Recall that an object $C$ in a Grothendieck category $\cat{G}$ is called \textbf{noetherian} if its lattice of subobjects satisfies the ascending chain condition. That is, there is no strictly ascending infinite chain of subobjects of $C$. It is equivalent to say that each subobject of $C$ is finitely generated. We then say the category $\cat{G}$ is \textbf{locally noetherian} if it has a generating set consisting of noetherian objects. We refer the reader to the classic reference~\cite[Section~V.4]{stenstrom} for basic information.  Note that every locally noetherian category is locally finitely generated.

\begin{proposition}\cite[Prop.~V.4.2.]{stenstrom}\label{prop-noetherian objects serre}
Let $\cat{G}$ be any Grothendieck category. Then the class of noetherian objects forms a Serre subcategory.
\end{proposition}

The next proposition contains standard characterizations of locally noetherian categories. We leave its proof as an exercise. In fact, a proof can be adapted by referring to the proof of Proposition~\ref{prop-locally coherent characterizations} below; just make necessary changes and refer to facts in~\cite[Section~V.4.]{stenstrom}.

\begin{proposition}\label{prop-locally noetherian characterizations}
The following are equivalent for any Grothendieck category $\cat{G}$.
\begin{enumerate}
\item $\cat{G}$ is locally noetherian.

\item $\cat{G}$ is locally finitely generated and the finitely generated objects coincide with the noetherian objects.

\item $\cat{G}$ is locally finitely generated and the full subcategory of finitely generated objects is a Serre subcategory.

\item Each $C \in \cat{G}$ is a direct union of noetherian objects.

\item $\cat{G}$ is locally finitely generated and the injective objects are closed under direct limits (or just under direct sums).
 \end{enumerate}
\end{proposition}

We show now that locally noetherian categories are particularly nice locally type $FP_{\infty}$ categories.
This theorem characterizes them in terms of the objects of type $FP_{\infty}$ and the absolutely clean objects.

\begin{theorem}\label{them-locally noetherian characterizations}
The following are equivalent for any Grothendieck category $\cat{G}$.
\begin{enumerate}
\item $\cat{G}$ is locally noetherian.

\item $\cat{G}$ is locally type $FP_{\infty}$ and the objects of type $FP_{\infty}$ coincide with the finitely generated objects.

\item $\cat{G}$ is locally type $FP_{\infty}$ and the objects of type $FP_{\infty}$ coincide with the noetherian objects.

\item $\cat{G}$ is locally type $FP_{\infty}$ and the objects of type $FP_{\infty}$ form a Serre subcategory.

\item $\cat{G}$ is locally type $FP_{\infty}$ and absolutely clean objects coincide with the injective objects.

\item $\cat{G}$ is locally finitely generated and each object is a direct union of subobjects of type $FP_{\infty}$.
\end{enumerate}
\end{theorem}

\begin{proof}
Say that $\cat{G}$ is locally noetherian. By Proposition~\ref{prop-locally noetherian characterizations}, the finitely generated objects coincide with the noetherian objects, and these form a Serre subcategory. So (2), (3), and (4) will each follow once we show that the the finitely generated objects are in fact of type $FP_{\infty}$. So let $F$ be finitely generated.
As a first step, we note that $F$ must be finitely presented. Indeed by~\cite[Prop.~V.3.4]{stenstrom} it is enough to show that for any morphism $B \xrightarrow{} F$ with $B$ finitely generated, its kernel is also finitely generated. But since in this case $B$ is noetherian, clearly $\ker{(B \xrightarrow{} F)}$ is finitely generated. Now we continue to show that $F$ is of type $FP_{\infty}$. It is left to show that the functors $\Ext^n_{\cat{G}}(F,-)$ preserve direct limits for all $n \geq 1$. But since $\cat{G}$ is locally noetherian, direct limits of injective objects are injective~\cite[page~124]{stenstrom}. So for any direct system $\{X_i\}_{i \in I}$, if we take injective coresolutions $X_i \hookrightarrow I_{X_i}$, then exactness of direct limits tells us that $\varinjlim X_i \hookrightarrow \varinjlim I_{X_i}$ is an injective coresolution. The canonical map $\xi_n : \varinjlim \Ext^n_{\cat{G}}(F,X_i) \xrightarrow{} \Ext^n_{\cat{G}}(F,\varinjlim X_i)$ can now easily be seen to be an isomorphism. Indeed the computation below holds since $F$ is finitely presented and direct limits commute with homology:
$$\varinjlim \Ext^n_{\cat{G}}(F,X_i) \cong \varinjlim H^n[\Hom(F,I_{X_i})] \cong  H^n[\Hom(F, \varinjlim I_{X_i})]  \cong \Ext^n_{\cat{G}}(F,\varinjlim X_i).$$ This completes the proof that the finitely generated objects coincide with the objects of type $FP_{\infty}$, and so (1) implies (2), (3), and (4).

Next, (3) implies (1) is immediate from definitions. We now show (2) implies (1). So say $\cat{G}$ is locally type $FP_{\infty}$, and let  $\{ G_i\}$ denote a generating set of objects of type $FP_{\infty}$. Then note that for any subobject $S \subseteq G_i$, we have $G_i/S$ is finitely generated and hence also of type $FP_{\infty}$, by hypothesis. So $\Ext^1_{\cat{G}}(G_i/S, -)$ preserves direct limits for all $S \subseteq G_i$. By Baer's criterion for injectivity, see~\cite[Prop.~V.2.9]{stenstrom}, it follows that direct limits of injectives are injective. So by Proposition~\ref{prop-locally noetherian characterizations} we conclude that $\cat{G}$ is locally noetherian.

To finish showing (1), (2), (3), and (4) are all equivalent, we will now show (4) implies (1). We are assuming $\cat{G}$ is locally type $FP_{\infty}$, so it suffices to show that each object of type $FP_{\infty}$ is noetherian. Certainly any type $FP_{\infty}$ object $F$ is finitely generated, and the hypothesis implies that each subobject is also finitely generated. This implies $F$ is noetherian.

We now turn to condition (5). Note that (2) implies (5), again using Bear's criterion for injectivity. Conversely, if (5) is true, then the injectives are closed under direct limits, and so (1) is true from Proposition~\ref{prop-locally noetherian characterizations}.

So it is left to tie in condition (6). Clearly, (2) implies (6). Conversely, suppose each object is a direct union of subobjects of type $FP_{\infty}$. We will show that all finitely generated objects are of type $FP_{\infty}$. So we take a finitely generated $F$, and write it as a direct union $F = \Sigma F_i$ where each $F_i$ is of type $FP_{\infty}$. Since it is a direct union and $F$ is finitely generated we have $F = F_{i_0}$ for some $i_0$.
\end{proof}

\subsection{Locally coherent categories}\label{subsec-locally coherent categories}
In the same way, we now show that all locally coherent categories are particularly nice locally type $FP_{\infty}$ categories. An object $C$ in a Grothendieck category $\cat{G}$ is called \textbf{coherent} if it is finitely presented and each finitely generated subobject is also finitely presented. The category $\cat{G}$ is called \textbf{locally coherent} if it has a generating set consisting of coherent objects. Such a category is clearly locally finitely presented and hence locally finitely generated. Locally noetherian categories are easily seen to be locally coherent.

We now recall some general facts about finitely presented objects and coherent objects in locally finitely generated Grothendieck categories. So assume $\cat{G}$ is locally finitely generated. Then by~\cite[Prop.~V.3.4]{stenstrom}, $F$ is finitely presented, meaning $\Hom_{\cat{G}}(F,-)$ preserves direct limits, if and only if for each short exact sequence  $0 \xrightarrow{} K \xrightarrow{} B \xrightarrow{} F \xrightarrow{} 0$ with $B$ finitely generated, then $K$ is also finitely generated. Using this characterization, it is an exercise to show that $F$ is finitely presented if and only if there exists a short exact sequence  $0 \xrightarrow{} K \xrightarrow{} F' \xrightarrow{} F \xrightarrow{} 0$ with $K$ finitely generated and $F'$ finitely presented. The following proposition also makes a very nice exercise. For a nicely presented proof we refer to~\cite[Proposition~1.5]{herzog-ziegler spectrum}. It is credited there as going back to~\cite[pp.~199]{auslander-coherent functors}.

\begin{proposition}\label{prop-coherent objects wide}
Let $\cat{G}$ be any locally finitely presented Grothendieck category. Then the class of coherent objects forms a wide subcategory.
\end{proposition}

\begin{corollary}\label{cor-thick-wide-serre-subcats}
Let $\cat{G}$ be any locally finitely presented Grothendieck category. Then the following hold.
\begin{enumerate}
\item The class of all objects of type $FP_{\infty}$ forms a thick subcategory.
\item The class of all coherent objects forms a wide, so in particular, an abelian, subcategory.
\item The class of all noetherian objects forms a Serre subcategory.
\end{enumerate}
\end{corollary}

The next proposition contains standard characterizations of locally coherent categories. It can all be found  scattered about the literature. In particular, see~\cite[Section~2]{roos-locally noetherian}, \cite[Theorem~1.6]{herzog-ziegler spectrum}, and \cite[App.~B]{stovicek-purity}.

\begin{proposition}\label{prop-locally coherent characterizations}
The following are equivalent for any Grothendieck category $\cat{G}$.
\begin{enumerate}
\item $\cat{G}$ is locally coherent.

\item $\cat{G}$ is locally finitely presented and the finitely presented objects coincide with the coherent objects.

\item $\cat{G}$ is locally finitely presented and the full subcategory of finitely presented objects is a wide subcategory.

\item Each $C \in \cat{G}$ is a direct limit of coherent objects.

\item $\cat{G}$ is locally finitely presented and the absolutely pure objects, that is, FP-injective objects, are closed under direct limits (or just under direct sums).
 \end{enumerate}
\end{proposition}

\begin{proof}
(1) implies (2). Let $\cat{G}$ be locally coherent and $F$ be finitely presented object. Then we can find an epimorphism $C \xrightarrow{g} F$ where $C$ is a finite direct sum of coherent objects. Note then that $C$ itself is coherent by Proposition~\ref{prop-coherent objects wide}. Since $F$ is finitely presented, we see that $\ker{g}$ must be finitely generated. A finitely generated subobject of a coherent object is clearly coherent, so $\ker{g}$ is coherent. It follows again from Proposition~\ref{prop-coherent objects wide} that $F$ is coherent. This proves (1) implies (2), and (2) implies (1) is clear from definitions.

(2) implies (3) is immediate from Proposition~\ref{prop-coherent objects wide}. We now prove (3) implies (2). Suppose the subcategory of finitely presented objects is wide. Then in particular it is abelian. Let $F$ be finitely presented and let $S \subseteq F$ be finitely generated. We wish to show $S$ is also finitely presented. But consider the short exact sequence $0 \xrightarrow{} S \xrightarrow{}  F \xrightarrow{} F/S \xrightarrow{} 0$. It shows $F/S$ must be finitely presented, so by hypothesis, $S$ must be too. This proves (3), and at this point we have shown (1) -- (3) are all equivalent.

(2) implies (4) follows from a standard fact about locally finitely presented categories. See~\cite[Theorem~1.11]{adamek-rosicky}. We now show (4) implies (1). Given $C \in \cat{G}$, write $C = \varinjlim C_i$ where each $C_i$ is coherent. Observe that we have an epimorphism $\oplus C_i \xrightarrow{} \varinjlim C_i$. Since the class of coherent objects has a small skeleton, this epimorphism shows that the ``set'' of all coherent objects generates $\cat{G}$.

This shows (1) -- (4) are all equivalent and condition (5) is also known to be equivalent. For example, see~\cite[Prop.~B.3]{stovicek-purity}.
\end{proof}

We now give the characterization of locally coherent categories in terms of objects of type $FP_{\infty}$ and absolutely clean objects.

\begin{theorem}\label{them-locally coherent characterizations}
The following are equivalent for any Grothendieck category $\cat{G}$.
\begin{enumerate}
\item $\cat{G}$ is locally coherent.

\item $\cat{G}$ is locally type $FP_{\infty}$ and the objects of type $FP_{\infty}$ coincide with the finitely presented objects.

\item $\cat{G}$ is locally type $FP_{\infty}$ and the objects of type $FP_{\infty}$ coincide with the coherent objects.

\item $\cat{G}$ is locally type $FP_{\infty}$ and the objects of type $FP_{\infty}$ form a wide subcategory.

\item $\cat{G}$ is locally type $FP_{\infty}$ and absolutely clean objects coincide with the absolutely pure objects.

\item $\cat{G}$ is locally finitely presented and each object is a direct limit of objects of type $FP_{\infty}$.
\end{enumerate}
\end{theorem}

\begin{proof}
Say that $\cat{G}$ is locally coherent. By Proposition~\ref{prop-locally coherent characterizations}, the finitely presented objects coincide with the coherent objects, and these form a wide subcategory. So (2), (3), and (4) will each follow once we show that the the finitely presented objects are in fact of type $FP_{\infty}$. So let $F$ be finitely presented. We wish to show that the functors $\Ext^n_{\cat{G}}(F,-)$ preserve direct limits for all $n \geq 1$. By Proposition~\ref{prop-locally coherent characterizations}, direct limits of absolutely pure objects are again absolutely pure. So given any direct system $\{X_i\}_{i \in I}$, if we take absolutely pure coresolutions $X_i \hookrightarrow A_{X_i}$, then exactness of direct limits tells us that $\varinjlim X_i \hookrightarrow \varinjlim A_{X_i}$ is again an absolutely pure coresolution. Moreover, taking $\class{S}$ in Lemma~\ref{lemma-stoviceks lemma} to be a set of isomorphism representatives for all finitely presented objects, part~(3) of that Lemma implies that
$\Ext^n_{\cat{G}}(F,A) = 0$ for all $n \geq 1$ and absolutely pure $A$.  In other words, absolutely pure objects are $\Hom_{\cat{G}}(F,-)$-acyclic, and it follows that we can compute $\Ext^n_{\cat{G}}(F,-)$ via absolutely pure coresolutions; see, for example, \cite[Theorem~XX.6.2]{lang}. So now we compute:
$$\varinjlim \Ext^n_{\cat{G}}(F,X_i) \cong \varinjlim H^n[\Hom(F,A_{X_i})] \cong  H^n[\Hom(F, \varinjlim A_{X_i})]  \cong \Ext^n_{\cat{G}}(F,\varinjlim X_i).$$
This means that the canonical map $\xi_n : \varinjlim \Ext^n_{\cat{G}}(F,X_i) \xrightarrow{} \Ext^n_{\cat{G}}(F,\varinjlim X_i)$ is an isomorphism and completes the proof that $F$ is type $FP_{\infty}$.

Next, (3) implies (1) is immediate from definitions. We now show (2) implies (1). So say $\cat{G}$ is locally type $FP_{\infty}$. By Proposition~\ref{prop-locally coherent characterizations}, it is enough to show that direct limits of absolutely pure objects are again absolutely pure. This would follow if it were true that $\Ext^1_\cat{G}(F, -)$ preserved direct limits for all finitely presented $F$. But this is true, since we are assuming that the finitely presented objects coincide with the $FP_{\infty}$ objects.

To finish showing (1), (2), (3), and (4) are all equivalent, we will now show (4) implies (1). We are assuming $\cat{G}$ is locally type $FP_{\infty}$, so it suffices to show that each object of type $FP_{\infty}$ is coherent. Certainly any type $FP_{\infty}$ object $F$ is finitely presented, so let $S \subseteq F$ be a finitely generated subobject. Part~(1) of Lemma~\ref{lemma-stoviceks lemma} shows that we can find an epimorphism $F' \twoheadrightarrow S$ where $F'$ is again of type $FP_{\infty}$. Now the composition $F' \twoheadrightarrow S \hookrightarrow F$ is a morphism between objects of type $FP_{\infty}$, and its image is $S$. The hypothesis implies that $S$ is also of type $FP_{\infty}$.

We now turn to condition (5). Note that (2) implies (5) by definitions. Conversely, if (5) is true, then the absolutely pure objects are closed under direct limits, and so (1) is true from Proposition~\ref{prop-locally coherent characterizations}.

So it is left to tie in condition (6). It is now clear that (1) implies (6). Conversely, suppose (6) holds. We will prove (2) by showing that any finitely presented object is type $FP_{\infty}$. Indeed let $F$ be finitely presented, and write $F = \varinjlim F_i$ where each $F_i$ is of type $FP_{\infty}$. Then $\Hom_{\cat{G}}(F,F) \cong \Hom_{\cat{G}}(F,\varinjlim F_i) \cong \varinjlim \Hom_{\cat{G}}(F, F_i)$.
This implies that the identiy map $1_F$ factors through some $F_i$. This in turn implies that $F$ is a direct summand of that $F_i$, and so $F$ is of type $FP_{\infty}$ by Proposition~\ref{prop-thickness of FP-infinity}.
\end{proof}

\section{The Inj and Abs clean model structures}

Let $\cat{G}$ be a Grothendieck category. This section has three parts. First, in Section~\ref{subsec-Inj model} we show that there is always a cofibrantly generated model structure on $\cha{G}$ whose homotopy category is equivalent to $K(Inj)$, the homotopy category of all complexes of injective objects. Following~\cite{bravo-gillespie-hovey} we call it the \emph{Inj model structure}.
In Section~\ref{subsec-Abs clean model} we consider the question of compact generation, and approach $K(Inj)$ through another model structure. We see that whenever $\cat{G}$ is locally type $FP_{\infty}$, the absolutely clean cotorsion pair of Proposition~\ref{prop-abs-clean-cotorsion-pair} lifts to a finitely generated model structure on $\cha{G}$ that we call the \emph{Abs clean model structure}.
Specializing to the case that $\cat{G}$ is locally noetherian, the Abs clean model structure coincides exactly with the Inj model structure, showing that $K(Inj)$ is compactly generated. When $\cat{G}$ is locally coherent it coincides with Stovicek's model structure from~\cite[Theorem~6.12]{stovicek-purity}, which he used to show that $K(Inj)$ is even compactly generated in the locally coherent case. In Section~\ref{subsec-AC-injectives}, we see that for a general locally type $FP_{\infty}$ category, the homotopy category of the Abs clean model structure is equivalent to the derived category of absolutely clean objects (with respect to its inherited Quillen exact structure). We denote it $\class{D}(\class{AC})$, and conclude it is a compactly generated triangulated category. Moreover, we show that it is equivalent to a full subcategory of $K(Inj)$ containing the DG-injective complexes.

\subsection{The Inj model structure}\label{subsec-Inj model}

Here we let $\cat{G}$ denote any Grothendieck category.  We recall, again, Baer's criterion for injectivity~\cite[Prop.~V.2.9]{stenstrom}. It says that we can test injectivity of an object using just the inclusions $C \hookrightarrow G_i$ where $G_i$ ranges through a generating set $\{G_i\}$ and $C$ ranges through the subobjects $C \subseteq G_i$. In the language of~\cite[Section~6]{hovey}, Baer's criterion translates to say that the canonical injective cotorsion pair $(\cat{G},\class{I})$ is a small cotorsion pair, with the set of all inclusions $C \hookrightarrow G_i$ serving as a set of \emph{generating monomorphisms}. In Hovey's correspondence between abelian model structures and cotorsion pairs, the small cotorsion pairs correspond to cofibrantly generated model structures. We use Hovey's notation and terminology from~\cite[Sections~2.1/7.4]{hovey-model-categories} regarding other aspects of cofibrantly generated model structures. In particular, given a set of maps $I$, we let $I$-inj denote the set of all maps possessing the \emph{right lifting property} with respect to maps in $I$.

We will encounter several injective model structures on $\cha{G}$ in this paper and the next lemma provides a set $I$ of \emph{generating cofibrations} for any of them. It is the set $J$ of \emph{generating trivial cofibrations} which varies in the different model structures.

\begin{lemma}\label{lemma-generating cofibrations}
Let $\cat{G}$ be any Grothendieck category and let $\tilclass{I}$ denote the class of all injective complexes in $\cha{G}$. That is, each $I \in \tilclass{I}$ is an exact complex with each $Z_nI$ injective. Then the injective cotorsion pair $(\class{A},\tilclass{I})$, where $\class{A}$ denotes the class of all complexes, is small in the sense of~\cite{hovey}. Explicitly, for any given generating set $\{G_i\}$, the set of generating monomorphisms can be taken to be
the set $$I = \{\, 0 \hookrightarrow D^n(G_i) \, \} \cup \{\, S^{n-1}(G_i) \hookrightarrow
D^n(G_i) \,\} \cup \{\, S^n(C) \hookrightarrow S^n(G_i) \,
\},$$ where $C \hookrightarrow G_i$ ranges through all inclusions of subobjects $C \subseteq G_i$.
Moreover, $I$-inj is precisely the class of all split epimorphisms with kernel in $\tilclass{I}$.
\end{lemma}

\begin{proof}
It follows from the above Baer's criterion and~\cite[Proposition~3.8]{gillespie-quasi-coherent} that this set will serve as a set of generating monomorphisms for $(\class{A},\tilclass{I})$. (The proof cited is sloppy and has a misstatement. But it is easily fixed by doing the second paragraph \emph{first}, and using the hypothesis $\Ext^1(S^n(G),X)=0$ to immediately deduce $X$ is exact.) Then it follows from~\cite[Lemma~3.2]{gillespie-quasi-coherent} that $I$-inj is how we describe. In short, the maps in $\{\, 0 \xrightarrow{} D^n(G_i) \, \}$ guarantee that everything in $I$-inj is an epimorphism, and then the maps in $\{\, S^{n-1}(G_i) \xrightarrow{}
D^n(G_i) \,\}$ guarantee that the kernel of such an epimorphism is an exact complex, and finally the maps in
$\{\, S^n(C) \xrightarrow{} S^n(G_i) \, \}$ guarantee that each cycle of this exact kernel is injective.
\end{proof}

\begin{theorem}\label{them-Inj model structure}
For any Grothendieck category $\cat{G}$ there is an abelian model structure on $\cha{G}$ that we call the \textbf{Inj model structure}. This is an injective model structure, meaning all objects are cofibrant and the trivially fibrant objects are the injective complexes. The fibrant objects are precisely the complexes of injectives. The model structure is cofibrantly generated. Explicitly, for any given generating set $\{G_i\}$, the generating cofibrations can be taken to be the set $I$ of Lemma~\ref{lemma-generating cofibrations}, while the generating trivial cofibrations can be taken to be the set $J_1 = \{\, D^n(C) \hookrightarrow D^n(G_i) \,\}$. The homotopy category of this model structure is equivalent to $K(Inj)$, the homotopy category of all complexes of injectives, and it is a well-generated triangulated category.
\end{theorem}

\begin{proof}
As in~\cite{gillespie-degreewise-model-strucs}, let $\dwclass{I}$ denote the class of all complexes which are ``degreewise'' injective. That is, $\dwclass{I}$ is the class of all complexes of injectives.
It follows from Baer's criterion and~\cite[Proposition~4.4]{gillespie-degreewise-model-strucs} that the pair $(\class{W}_1,\dwclass{I})$, where $\class{W}_1=\leftperp{\dwclass{I}}$, is a small cotorsion pair with $J_1$ serving as the generating monomorphisms. We note that $\class{W}_1$ contains the generating set $\{D^n(G_i)\}$, and that $(\class{W}_1,\dwclass{I})$ must be a functorially complete cotorsion pair by~\cite[Theorem~6.5]{hovey}. Once we show $\class{W}_1$ is thick and that $\class{W}_1 \cap \dwclass{I} = \tilclass{I}$, then~\cite[Lemma~6.7]{hovey} guarantees that the two cotorsion pairs $(\class{A},\tilclass{I})$ and $(\class{W}_1,\dwclass{I})$ determine a cofibrantly generated model structure on $\cha{G}$ with $I$ being a set of generating cofibrations and $J_1$ being a set of generating trivial cofibrations. But the class $\class{W}_1$ is thick since in this case a complex $X \in \leftperp{\dwclass{I}}$ if and only if $\homcomplex(X,I)$ is exact for all $I \in \dwclass{I}$. So the retracts and two out of three argument from~\cite[Theorem~4.1]{bravo-gillespie-hovey} holds in the same way, and $\class{W}_1$ contains the contractible complexes. Since injective complexes are contractible we have $\tilclass{I} \subseteq \class{W}_1$ and so from~\cite[Proposition~3.3]{bravo-gillespie-hovey} we conclude $\class{W}_1 \cap \dwclass{I} = \tilclass{I}$. From the fundamental theorem of model categories we know that the homotopy category of this model structure is equivalent to $\dwclass{I}/\sim$ where $\sim$ denotes the formal homotopy relation. However, it follows from~\cite[Corollary~4.8]{gillespie-exact model structures} that $f \sim g$ if and only if $g-f$ factors through an injective. Since injective complexes are contractible this implies that $f \sim g$ if and only if $f$ and $g$ are chain homotopic in the usual sense. So the homotopy category is just $K(Inj)$.

Since we have a cofibrantly generated model structure on a locally presentable (pointed) category, a main result from~\cite{rosicky-brown representability combinatorial model srucs} assures us that it is well generated in the sense of~\cite{neeman-well generated}. Also see~\cite[Section~7.3]{hovey-model-categories}.
\end{proof}

\subsection{The Abs clean model structure}\label{subsec-Abs clean model}

We now let $\cat{G}$ denote any locally type $FP_{\infty}$ category, and consider the question of compact generation. Krause showed in~\cite{krause-stable derived cat of a Noetherian scheme} that $K(Inj)$ is compactly generated whenever $\cat{G}$ is locally noetherian and this has been extended by Stovicek in~\cite{stovicek-purity} to the case of $\cat{G}$ locally coherent. This is all related to the fact that the \emph{Abs clean model structure} that we construct in this section, is always a finitely generated model structure.

Referring to Proposition~\ref{prop-abs-clean-cotorsion-pair}, we have the hereditary cotorsion pair $(\leftperp{\class{AC}},\class{AC})$, where $\class{AC}$ is the class of absolutely clean objects. Recall that it is small and that the set of generating monomorphisms can be taken to be the set of all monomorphisms $F' \hookrightarrow F''$ that fit into a short exact sequence $0 \xrightarrow{} F' \xrightarrow{} F'' \xrightarrow{} F \xrightarrow{} 0$ with $F'$, $F''$, and $F$ each in $FP_{\infty}(\cat{G})$. Here $FP_{\infty}(\cat{G})$ is a set of isomorphism representatives for the class of all objects of type $FP_{\infty}$. We use this notation in the following lemma and theorem.

\begin{lemma}\label{lemma-generating cofibrations for Abs clean}
Let $\cat{G}$ be a locally type $FP_{\infty}$ category and let $\tilclass{AC}$ denote the class of all absolutely clean complexes in $\cha{G}$. That is, each $A \in \tilclass{AC}$ is an exact complex with each $Z_nA$ absolutely clean. Then $(\leftperp{\tilclass{AC}},\tilclass{AC})$ is a cotorsion pair, and small in the sense of~\cite{hovey}. Explicitly, given a generating set $\{G_i\} \subseteq FP_{\infty}(\cat{G})$, the generating monomorphisms can be taken to be the set $$I' = \{\, 0 \hookrightarrow D^n(G_i) \, \} \cup \{\, S^{n-1}(G_i) \hookrightarrow
D^n(G_i) \,\} \cup \{\, S^n(F') \hookrightarrow S^n(F'') \,
\},$$
where $F' \hookrightarrow F''$ ranges through all monomorphisms that fit into a short exact sequence  $0 \xrightarrow{} F' \xrightarrow{} F'' \xrightarrow{} F \xrightarrow{} 0$ with $F'$, $F''$, and $F$ each in $FP_{\infty}(\cat{G})$.
Moreover, $I'$-inj is precisely the class of all epimorphisms with kernel in $\tilclass{AC}$.
\end{lemma}

\begin{remark}
By definition, a complex $A \in \tilclass{AC}$ is an exact complex with each $Z_nA$ absolutely clean. We will simply call them \textbf{absolutely clean complexes}.  It is explicitly shown in~\cite[Proposition~2.6]{bravo-gillespie}, that at least for modules over a ring $R$, $\tilclass{AC}$ coincides with the class of \emph{categorically} absolutely clean objects in $\ch$.
So note that Lemma~\ref{lemma-generating cofibrations for Abs clean} is the absolutely clean analog of Lemma~\ref{lemma-generating cofibrations}. It will play the same roll in that it will provide generating cofibrations for more than one ``absolutely clean'' model structure on $\cha{G}$.
\end{remark}

\begin{proof}
Just like Lemma~\ref{lemma-generating cofibrations}, it follows from~\cite[Proposition~3.8]{gillespie-quasi-coherent}. We instead just start with the cotorsion pair $(\leftperp{\class{AC}},\class{AC})$ of Proposition~\ref{prop-abs-clean-cotorsion-pair}.
\end{proof}

\begin{theorem}\label{them-Abs clean model struc}
Let $\cat{G}$ be a locally type $FP_{\infty}$ category. Then there is an hereditary abelian model structure on $\cha{G}$ that we call the \textbf{Abs clean model structure}. The class of fibrant objects, denoted $\dwclass{AC}$, is the class of all complexes of absolutely clean objects. The class $\tilclass{AC}$ of absolutely clean complexes is precisely the class of trivially fibrant objects. The model structure is finitely generated and so the homotopy category is compactly generated. Explicitly, given a generating set $\{G_i\} \subseteq FP_{\infty}(\cat{G})$, the generating cofibrations can be taken to be the set $I'$ of Lemma~\ref{lemma-generating cofibrations for Abs clean}, while
the generating trivial cofibrations can be taken to be the set $J'_1 = \{\, D^n(F') \hookrightarrow D^n(F'') \,\}$.
Here, again $F' \hookrightarrow F''$ ranges through all monomorphisms that fit into a short exact sequence  $0 \xrightarrow{} F' \xrightarrow{} F'' \xrightarrow{} F \xrightarrow{} 0$ with $F'$, $F''$, and $F$ each in $FP_{\infty}(\cat{G})$.
\end{theorem}

\begin{proof}
As in the proof of the Inj model structure we again use the notation of~\cite{gillespie-degreewise-model-strucs}. This time $\dwclass{AC}$ denotes the class of all complexes which are ``degreewise'' absolutely clean. That is, $\dwclass{AC}$ is the class of all complexes of absolutely clean objects. Combining Proposition~\ref{prop-abs-clean-cotorsion-pair} with~\cite[Proposition~4.4]{gillespie-degreewise-model-strucs}, we immediately get a small cotorsion pair $(\leftperp{\dwclass{AC}},\dwclass{AC})$ with the described set $J'_1$ being exactly the generating monomorphisms. It is easy to see that this cotorsion pair is hereditary since $(\leftperp{\class{AC}},\class{AC})$ is hereditary. By~\cite[Proposition~3.2]{gillespie-degreewise-model-strucs}, we see that $\leftperp{\dwclass{AC}}$ consists precisely of the complexes $X$ such that each $X_n \in \leftperp{\class{AC}}$, and such that any chain map $X \xrightarrow{} A$ with $A \in \dwclass{AC}$, is null homotopic. From this, one can argue that $\leftperp{\dwclass{AC}} \cap \dwclass{AC}$ coincides with the class of contractible complexes with components in $\leftperp{\class{AC}} \cap \class{AC}$.

The other cotorsion pair is $(\leftperp{\tilclass{AC}},\tilclass{AC})$ from Lemma~\ref{lemma-generating cofibrations for Abs clean}.  Again, we can argue that this cotorsion pair is hereditary and that $\leftperp{\tilclass{AC}} \cap \tilclass{AC}$ coincides with the class of contractible complexes with components in $\leftperp{\class{AC}} \cap \class{AC}$. (Note that $\leftperp{\tilclass{AC}} = \dgclass{(\leftperp{\class{AC}})}$ in the notation of~\cite{gillespie}.)

Since $(\leftperp{\dwclass{AC}},\dwclass{AC})$ and $(\leftperp{\tilclass{AC}},\tilclass{AC})$ are each complete hereditary cotorsion pairs satisfying $\leftperp{\dwclass{AC}} \cap \dwclass{AC} = \leftperp{\tilclass{AC}} \cap \tilclass{AC}$, the existence of a unique hereditary abelian model structure as described is now automatic from~\cite{gillespie-hovey triples}. That is, there exists a unique thick class $\class{V}_1$ such that $(\leftperp{\tilclass{AC}}, \class{V}_1, \dwclass{AC})$ is an hereditary Hovey triple.  In particular, $\dwclass{AC}$ is the class of fibrant objects and  $\class{V}_1 \cap \dwclass{AC} = \tilclass{AC}$ is the class of trivially fibrant objects. By~\cite[Lemma~6.7]{hovey}, the model structure is cofibrantly generated with $I'$ serving as a set of generating cofibrations and $J'_1$ serving as a set of generating trivial cofibrations. But since the domains and codomains of maps in $I'$ and $J'_1$ are all finitely presented, we see that the model structure is in fact \emph{finitely generated} in the sense of~\cite[Section~7.4]{hovey}. Now, Hovey showed in~\cite[Corollary~7.4.4]{hovey} that the cokernels from all maps in $I'$ form a set of compact weak generators for the homotopy category. In other words, the homotopy category is compactly generated.
\end{proof}

We note that in the case that $\cat{G}$ is locally noetherian, the Abs clean model structure coincides with the Inj model structure of Theorem~\ref{them-Inj model structure}. This gives the following corollary which recovers a result of Krause from~\cite{krause-stable derived cat of a Noetherian scheme}.

\begin{corollary}\label{cor-compact generation of K(Inj)}
Let $\cat{G}$ be locally noetherian. Then $K(Inj)$, the homotopy category of all injective complexes, is compactly generated.
\end{corollary}

\begin{remark}
Remarkably, Stovicek has extended the above result to the locally coherent case. We note that in this case, the Abs clean model structure coincides with Stovicek's model structure from~\cite[Theorem~6.12]{stovicek-purity}. We refer the reader to~\cite{stovicek-purity} for full details on the compact generation of $K(Inj)$ in the locally coherent case.
\end{remark}

\subsection{The derived category of absolutely clean objects and the AC-injective model structure}\label{subsec-AC-injectives}
Again, we are considering the general case of an arbitrary locally type $FP_{\infty}$ category. Note that the full additive subcategory $\class{AC}$, of absolutely clean objects, is closed under extensions and direct summands. So it naturally inherits the structure of a weakly idempotent complete exact category (WIC exact category); for example, see~\cite[Lemma~5.1]{gillespie-exact model structures}. The short exact sequences are just the usual short exact sequences, but with all three terms in $\class{AC}$. By the \emph{derived category of absolutely clean objects}, denoted $\class{D}(\class{AC})$, we mean the derived category with respect to this exact structure, in the sense of~\cite{neeman-exact category} and~\cite{keller-derived cats}. Naively, it is the triangulated category obtained from $\cha{AC}$, the category of chain complexes of absolutely clean objects, by killing the exact complexes. Note that an exact complex with respect to the exact structure on $\class{AC}$ is precisely a complex in $\tilclass{AC}$. Theorem~\ref{them-Abs clean model struc} immediately implies the following.

\begin{corollary}\label{cor-compact generation of abs clean derived cat}
Let $\cat{G}$ be any locally type $FP_{\infty}$ category. Then $\class{D}(\class{AC})$, the derived category of absolutely clean objects, is a compactly generated triangulated category.
\end{corollary}

\begin{proof}
It was shown in~\cite{gillespie-exact model structures} that a WIC exact category has enough structure to consider ``abelian'' model structures, which in this generality are called ``exact'' model structures.
From~\cite[Proposition~5.2]{gillespie-exact model structures}, the category $\cha{AC}$ inherits an exact model structure from the Hovey triple $(\leftperp{\tilclass{AC}}, \class{V}_1, \dwclass{AC})$ of Theorem~\ref{them-Abs clean model struc}. Precisely, it is the Hovey triple  $(\leftperp{\tilclass{AC}} \cap \dwclass{AC}, \class{V}_1 \cap \dwclass{AC}, \dwclass{AC})$, and these are cotorsion pairs in $\cha{AC}$ with its naturally inherited exact structure. Since the class of trivial objects is $\class{V}_1 \cap \dwclass{AC} = \tilclass{AC}$, this exact model structure has homotopy category equivalent to $\class{D}(\class{AC})$. But by~\cite[Corollary~5.4]{gillespie-exact model structures}, the homotopy category of this restricted exact model structure is equivalent to the original one, proving the corollary.
\end{proof}

%%%%%%%%%%%%
We have just seen that the homotopy category of the Abs clean model structure, which is equivalent to $\class{D}(\class{AC})$, is always compactly generated. In the locally coherent case, $\class{D}(\class{AC})$ becomes the derived category of absolutely pure objects, and Stovicek shows in~\cite{stovicek-purity} that this category is equivalent to $K(Inj)$. Next we prove that, in general, $\class{D}(\class{AC})$ is equivalent to a full subcategory of $K(Inj)$ containing the DG-injective complexes. Recall that a complex of injectives $I$ is called \textbf{DG-injective} if it has the property that all chain maps $E \xrightarrow{} I$, with $E$ an exact complex, are null homotopic.

\begin{definition}\label{def-AC-injective complexes}
Let $\cat{G}$ be a locally type $FP_{\infty}$ category. Call a chain complex $I$ of injectives \textbf{AC-injective} if it has the property that all chain maps $A \xrightarrow{} I$, with $A \in \tilclass{AC}$, are null homotopic. Recall that $\tilclass{AC}$ denotes the class of all exact complexes with absolutely clean cycles.
\end{definition}

Note that any DG-injective complex is automatically AC-injective.

\begin{theorem}\label{them-injective model for abs clean}
Let $\cat{G}$ be a locally type $FP_{\infty}$ category. Then there is a cofibrantly generated abelian model structure on $\cha{G}$, that we call the \textbf{AC-injective model structure},  as follows:
\begin{enumerate}
\item The model structure is injective, meaning all objects are fibrant.
\item The class $\class{F}$ of fibrant objects is the class of AC-injective complexes.
\item The class $\class{V}_1$ of trivial objects is $\leftperp{\class{F}}$.
\item The class $\class{V}_1$ of trivial objects coincides with the class of trivial objects in the Abs clean model structure. Therefore, their homotopy categories coincide.
\end{enumerate}
The homotopy category of the AC-injective model structure is equivalent to the chain homotopy category $K(\textnormal{AC-}Inj)$, of all AC-injective complexes.
\end{theorem}

\begin{proof}
Let $\{G_i\}$ be a generating set. Letting $C \hookrightarrow G_i$ range through all possible inclusions, set $\class{S} = \{D^n(G_i/C)\} \cup \{A_{\alpha}\}$. Here, $\{A_{\alpha}\}$ is a set of complexes in $\tilclass{AC}$ which generates all others as transfinite extensions. See Proposition~\ref{prop-abs-clean-deconstructible} or~\cite[Lemma~4.3 and~Prop.~4.11]{gillespie-quasi-coherent} or~\cite[Theorem~4.2]{stovicek} for the existence of such a set $\{A_{\alpha}\}$. Let $\class{F}$ be the alleged class of fibrant objects in the statement of the theorem. Using Eklof's lemma, one can deduce that $\rightperp{\class{S}} = \class{F}$. Therefore, by~\cite[Corollary~2.14(2)]{saorin-stovicek}, we have that $(\leftperp{\class{F}},\class{F})$ is a complete cotorsion pair. (It is automatically small too, but we need not describe the generating monomorphisms for our purposes here.) Set $\class{V}_1 = \leftperp{\class{F}}$. Since $\class{F}$ consists of complexes of injectives we see that $\class{V}_1$ is the class of all complexes $W$ such that $\homcomplex(W,I)$ is exact for all $I \in \class{F}$. Similar to Theorem~\ref{them-Inj model structure}, we can argue that $\class{V}_1$ is thick and contains all injective complexes. Hence $(\class{V}_1,\class{F})$ is an injective cotorsion pair by~\cite[Proposition~3.3]{bravo-gillespie-hovey}. That is, (1), (2), and (3) all hold by Hovey's correspondence.

We now prove (4). We use~\cite{gillespie-hovey triples} which says that the thick class $\class{W}$ in any Hovey triple $(\class{Q},\class{W},\class{R})$ is unique. That is, there can only be one thick class $\class{V}_1$ making $(\class{Q},\class{V}_1,\class{R})$ is a Hovey triple.  So in the current case, with $\class{V}_1 = \leftperp{\class{F}}$, we need to show $\leftperp{\tilclass{AC}} \cap \class{V}_1 = \leftperp{\dwclass{AC}}$ and $\class{V}_1 \cap \dwclass{W} = \tilclass{AC}$.

We first show $\class{V}_1 \cap \dwclass{W} = \tilclass{AC}$, by following the method of Stovicek from~\cite[Proposition~6.11]{stovicek-purity}. Our above Propositions~3.9 and~3.10 say that the class $\class{AC}$ is \emph{deconstructible} in the sense of~\cite{stovicek-exact models}. So by~\cite[Theorem~3.16]{stovicek-exact models} it inherits the structure of an exact category of Grothendieck type. In particular, it has enough injectives, and it is easy to see that these are precisely the usual injectives from the ambient category $\cat{G}$. Moreover, by~\cite[Lemma~7.9 and~Theorem~7.11]{stovicek-exact models} we get that $\cha{AC}$, with its inherited degreewise exact structure, has an ``injective'' model structure. It is represented by the cotorsion pair $(\tilclass{AC}, \rightperp{\tilclass{AC}})$, and we emphasize that this is a cotorsion pair \emph{in} the exact category $\cha{AC}$.
We claim that $\rightperp{\tilclass{AC}} = \class{F}$. First, if $X \in \class{F}$, then since $X$ is a complex of injectives, any Yoneda Ext group $\Ext^1_{\cha{AC}}(A,X)$ coincides with the subgroup of all ``degreewise split'' extensions. Since any chain map $A \xrightarrow{} X$ with $A \in \tilclass{AC}$ must be null homotopic, this implies that $\Ext^1_{\cha{AC}}(A,X) = 0$ for all $A \in \tilclass{AC}$. So $\class{F} \subseteq \rightperp{\tilclass{AC}}$.
For the reverse containment, suppose $X \in \cha{AC}$ is in $\rightperp{\tilclass{AC}}$. Then for any absolutely clean object $A \in \class{AC}$, we have $0 = \Ext^1_{\cha{AC}}(D^n(A),X) \cong \Ext^1_{\class{AC}}(A,X_n)$. This means that each $X_n$ is injective in the exact category $\class{AC}$, which as pointed out above, means that each $X_n$ is $\cat{G}$-injective. It now follows that $X \in \class{F}$.

On the other hand, we have the cotorsion pair $(\class{V}_1,\class{F})$, in $\cha{G}$. A straightforward checking shows that it restricts to a cotorsion pair $(\class{V}_1 \cap \dwclass{AC}, \class{F})$ \emph{in} the exact category $\cha{AC}$. To summarize, we have two cotorsion pairs in $\cha{AC}$. They are $(\tilclass{AC}, \class{F})$ and $(\class{V}_1 \cap \dwclass{AC}, \class{F})$. Since their right sides are the same, so must be their left sides. That is, $\class{V}_1 \cap \dwclass{W} = \tilclass{AC}$.

Now $\leftperp{\tilclass{AC}} \cap \class{V}_1 = \leftperp{\dwclass{AC}}$ is automatically true by Lemma~\ref{lemma-half related}(2). This completes the proof.
\end{proof}

\section{A model for Krause's stable derived category and recollement}\label{sec-recollements}

Becker showed in~\cite{becker} that Krause's recollement $S(R) \xrightarrow{} K(Inj) \xrightarrow{} \class{D}(R)$, from~\cite{krause-stable derived cat of a Noetherian scheme}, holds for any ring $R$, even without the noetherian hypothesis. Here $S(R)$ is the full subcategory of $K(Inj)$ consisting of all exact complexes. Krause's original work in~\cite{krause-stable derived cat of a Noetherian scheme} was in the setting of separated noetherian schemes and locally noetherian categories. In practice, such categories quite often come with a set of generators of finite projective dimension. As the author indicated at~\cite{gillespie-ASTA talk}, this hypothesis is connected to the problem of obtaining the recollement, and that is the main theme of this section. We would like to know when a recollement holds and when all three categories in the recollement are compactly generated.

\subsection{Locally finite dimensional categories}\label{subsec-generators of finite projective dimension}
We first will consider Grothendieck categories coming with a set of generators of finite projective dimension. The author first learned of this hypothesis by reading Hovey's~\cite{hovey sheaves} where he used it to construct certain model structures on complexes of sheaves. Later, the current author made systematic use of this hypothesis to lift cotorsion pairs in Grothendieck categories to model structures on the associated chain complex category~\cite{gillespie-degreewise-model-strucs}.

Recall that Grothendieck categories may not have enough projectives. However, $\Ext^n_{\cat{G}}(A,B)$ is still always defined and can be computed using injective resolutions. We say that an object $A$ has \emph{finite projective dimension} if for any object $B$ there is an integer $n$ for which $\Ext^i_n(A,B) = 0$ for all $i>n$. For convenience, we will say a Grothendieck category $\cat{G}$ is \textbf{locally finite dimensional} if it possesses a generating set $\{G_i\}$ for which each $G_i$ has finite projective dimension. If furthermore, each $G_i$ is of type $FP_{\infty}$ we will say $\cat{G}$ is \textbf{locally  finite dimensionally type $\boldsymbol{FP_{\infty}}$}. If the $G_i$ are coherent, it is \textbf{locally finite dimensionally coherent}; if they are noetherian, \textbf{locally finite dimensionally noetherian}.
Building on the list of examples from Section~\ref{subsec-Examples of FP-infinity categories}, we have the following motivating examples.

\begin{example}
Following up on Example~\ref{example-modules}, the generating set $\{{}_RR\}$ shows the category $R$-Mod to be locally finite dimensionally type $FP_{\infty}$. It is locally finite dimensionally coherent (resp. noetherian) precisely when the ring $R$ is left coherent (resp. noetherian). Note that this all generalizes in an obvious way to any Grothendieck category with a set of finitely generated projective generators.
\end{example}

\begin{example}
Following up on Example~\ref{example-chain complexes}, we see that $\cha{G}$ is locally finite dimensionally type $FP_{\infty}$ whenever $\cat{G}$ is such.
\end{example}

\begin{example}
Again, let $(X,\class{O}_X)$ be a ringed space where the underlying space $X$ is a finite dimensional compact manifold. Then $\class{O}_X$-Mod is a locally finite dimensional Grothendieck category.
This holds more generally when $X$ is a finite dimensional compact manifold that is countable at infinity. This example is taken from~\cite[Prop.~3.3]{hovey sheaves}.
\end{example}

\begin{example}
Following up on Example~\ref{example-sheaves}, let $(X,\class{O}_X)$ be a quasi-noetherian ringed space; for example $X$ could be a concentrated scheme. We have seen that the category $\class{O}_X$-Mod of all sheaves of $\class{O}_X$-modules is a locally type $FP_{\infty}$ category. If $X$ is a finite dimensional noetherian scheme, then $\class{O}_X$-Mod is a locally finite dimensionally noetherian category. This follows from Grothendieck's vanishing theorem~\cite[Theorem~2.7]{hartshorne}, since an open subspace of a finite dimensional noetherian space is again a finite dimensional noetherian space. Again, see~\cite[Prop.~3.3]{hovey sheaves}.
\end{example}

\begin{example}
Let Qco($X$) denote the category of quasi-coherent sheaves on $X$, where $X$ is a separated noetherian scheme with a family of ample line bundles. In this case, locally free sheaves of finite rank are generators of finite projective dimension. See~\cite[Prop.~2.3]{hovey sheaves} and also~\cite[Example~4.8]{krause-stable derived cat of a Noetherian scheme}.
Since Qco($X$) is a locally noetherian category we see that it is a locally finitely dimensionally noetherian category. One can check directly that for a locally free sheaf $F$ of finite rank, $\Hom_{\text{Qco(X)}}(F,-)$ preserves direct limits, since direct limits are taken in the category of presheaves when the underlying space $X$ is noetherian~\cite[Ex.~II.1.11]{hartshorne}
\end{example}

\begin{example}
Let Qco($X$) denote the category of quasi-coherent sheaves on $X$, where $X$ is a quasi-projective scheme with coherent structure sheaf. Then as Stovicek points out in~\cite[pp.~31]{stovicek-purity}, the
category Qco($X$) is a locally finite dimensionally coherent category.
\end{example}

\subsection{The exact Inj model structure}\label{subsec-exact Inj model structure}

Here we generalize the exact Inj model structure from~\cite{bravo-gillespie-hovey} to the setting of any locally finite dimensional Grothendieck category. It will show that the stable derived category $S(\cat{G})$ is at least always well generated in this case. Becker's method extends to immediately obtain Krause's recollement in this setting. 

\begin{theorem}\label{them-exact Inj model structure}
Let $\cat{G}$ be a locally finite dimensional Grothendieck category.  Then there is an abelian model structure on $\cha{G}$ that we call the \textbf{exact Inj model structure}. This is an injective model structure, meaning all objects are cofibrant and the trivially fibrant objects are the injective complexes. The fibrant objects are precisely the exact complexes of injectives. The model structure is cofibrantly generated. Explicitly, given a generating set $\{G_i\}$ with each $G_i$ of finite projective dimension, the generating cofibrations can be taken to be the set $I$ of Lemma~\ref{lemma-generating cofibrations}. The generating trivial cofibrations can be taken to be the set $$J_2 = \{\, D^n(C) \hookrightarrow D^n(G_i) \,\} \cup \{\, S^{n-1}(G_i) \hookrightarrow D^n(G_i) \,\},$$
where again $C \hookrightarrow G_i$ ranges through all inclusions of subobjects $C \subseteq G_i$. We note that the class $\class{W}_2$ of trivial objects contains all contractible complexes and that the homotopy category of this model structure is equivalent to $S(\cat{G})$, the stable derived category of Krause.
\end{theorem}

\begin{proof}
As in~\cite{gillespie-degreewise-model-strucs}, let $\exclass{I}$ denote the class of all exact complexes of injectives.
It follows from Baer's criterion and~\cite[Proposition~4.6]{gillespie-degreewise-model-strucs} that the pair $(\class{W}_2,\exclass{I})$, where $\class{W}_2 = \leftperp{\exclass{I}}$, is a small cotorsion pair with $J_2$ serving as the generating monomorphisms. (This is where the finite projective dimension hypothesis on the generators $G_i$ is used.) We note that $\class{W}_2$ contains the generating set $\{D^n(G_i)\}$, and that $(\class{W}_2,\exclass{I})$ must also be a functorially complete cotorsion pair by~\cite[Theorem~6.5]{hovey}. The remaining statements are proved in the exact same way that we proved Theorem~\ref{them-Inj model structure}. Simply replace $\class{W}_1$, $\dwclass{I}$, and $K(Inj)$ in that proof with $\class{W}_2$, $\exclass{I}$, and $S(\cat{G})$.
\end{proof}

Becker's method from~\cite{becker} will now apply to show that the recollement of Krause holds. In fact, let $(\class{E}, \dgclass{I})$ denote the cotorsion pair where $\class{E}$ is the class of exact complexes and $\dgclass{I}$ is the class of DG-injective complexes. Then the three cotorsion pairs $$(\class{W}_1,\dwclass{I}) \ , \ \ \ \ (\class{W}_2,\exclass{I}) \ , \ \ \ \ (\class{E}, \dgclass{I})$$
satisfy the hypotheses of~\cite[Theorem~4.6]{gillespie-recollements} and immediately give the following corollary.

\begin{corollary}\label{cor-recollement of krause}
Let $\cat{G}$ be a locally finite dimensional category. Then the canonical functors $S(\cat{G}) \xrightarrow{} K(Inj) \xrightarrow{} \class{D}(\cat{G})$ induce a recollement.
\end{corollary}

\subsection{The exact Abs clean model structure}\label{subsec-exact Abs clean model structure}

In the spirit of Section~\ref{subsec-Abs clean model} we would like a version of the exact Inj model structure which is always compactly generated. We are led to the following.

\begin{theorem}\label{them-exact Abs clean model structure}
Let $\cat{G}$ be a locally finite dimensionally $FP_{\infty}$ category.  Then there is an hereditary abelian model structure on $\cha{G}$ that we call the \textbf{exact Abs clean model structure}.
The class of fibrant objects, denoted $\exclass{AC}$, is the class of all exact complexes of absolutely clean objects. The class of trivially fibrant objects is the class $\tilclass{AC}$ of all absolutely clean complexes; the exact complexes with absolutely clean cycles. The model structure is finitely generated and so the homotopy category is compactly generated.
Explicitly, given a generating set $\{G_i\} \subseteq FP_{\infty}(\cat{G})$ with each $G_i$ of finite projective dimension, the generating cofibrations can be taken to be the set $I'$ of Lemma~\ref{lemma-generating cofibrations for Abs clean}.
The generating trivial cofibrations can be taken to be the set $$J'_2 = \{\, D^n(F') \hookrightarrow D^n(F'') \,\} \cup \{\, S^{n-1}(G_i) \hookrightarrow D^n(G_i) \,\},$$
where again $F' \hookrightarrow F''$ ranges through all monomorphisms that fit into a short exact sequence  $0 \xrightarrow{} F' \xrightarrow{} F'' \xrightarrow{} F \xrightarrow{} 0$ with $F'$, $F''$, and $F$ each in $FP_{\infty}(\cat{G})$.
\end{theorem}

\begin{proof}
We again apply~\cite[Proposition~4.6]{gillespie-degreewise-model-strucs} but to the small cotorsion pair $(\leftperp{\class{AC}},\class{AC})$ of Proposition~\ref{prop-abs-clean-cotorsion-pair}. This gives us a small cotorsion pair $(\leftperp{\exclass{AC}},\exclass{AC})$ with the described set $J'_2$ serving as the generating monomorphisms. (Again, the cited proposition uses the finite projective dimension hypothesis on the generators.) The rest of the theorem follows by imitating the proof of Theorem~\ref{them-Abs clean model struc} to wind up with a Hovey triple $(\leftperp{\tilclass{AC}}, \class{V}_2, \exclass{AC})$.
\end{proof}

For the following corollary, let $\class{D}(\class{AC})$ denote the homotopy category of the Abs clean model structure of Theorem~\ref{them-Abs clean model struc}, and let $S(\class{AC})$ denote the homotopy category of the exact Abs clean model structure of Theorem~\ref{them-exact Abs clean model structure}.

\begin{corollary}\label{cor-FP-infinity recollement of krause}
Let $\cat{G}$ be a locally finite dimensionally type $FP_{\infty}$ category. Then there is a recollement $S(\class{AC}) \xrightarrow{} \class{D}(\class{AC}) \xrightarrow{} \class{D}(\cat{G})$, and all three are compactly generated triangulated categories.
\end{corollary}

\begin{proof}
The Abs clean model structure of Theorem~\ref{them-Abs clean model struc} is represented by the Hovey triple $\class{M}_1 = (\leftperp{\tilclass{AC}}, \class{V}_1, \dwclass{AC})$. The exact Abs clean model structure of Theorem~\ref{them-exact Abs clean model structure} is represented by the Hovey triple $\class{M}_2 = (\leftperp{\tilclass{AC}}, \class{V}_2, \exclass{AC})$.
We now appeal to~\cite[Theorem~4.7]{gillespie-degreewise-model-strucs}. This theorem immediately provides a Hovey triple $\class{M}_3 = (\leftperp{\exclass{AC}}, \class{E}, \dwclass{AC})$ where $\class{E}$ is the class of exact complexes. So its homotopy category is the usual derived category $\class{D}(\cat{G})$.  Moreover, $J'_2$ is the set of generating cofibrations for $\class{M}_3$ while $J'_1$ is the set of generating trivial cofibrations for $\class{M}_3$.  Since the domains and codomains of all maps in $J'_1$ and $J'_2$ are finitely presented, this means that $\class{D}(\cat{G})$ is also compactly generated. The existence of a recollement is immediate from~\cite[Theorem~3.4]{gillespie-mock projectives}.
\end{proof}

\subsection{The exact AC-injective model structure}
Still, we are letting $\cat{G}$ be a locally finite dimensionally type $FP_{\infty}$ category.
We have just seen that the homotopy category of the exact Abs clean model structure is compactly generated. Whenever $\cat{G}$ is locally coherent, Ho($\class{M}$) will be equivalent to $S(\cat{G})$, the homotopy category of all exact complexes of injectives. In the general case, we have in the spirit of Theorem~\ref{them-injective model for abs clean}, the following injective model for the exact Abs clean model structure.

\begin{theorem}\label{them-injective model for exact abs clean}
Let $\cat{G}$ be a locally finite dimensionally type $FP_{\infty}$ category. Then there is a cofibrantly generated abelian model structure on $\cha{G}$, that we call the \textbf{exact AC-injective model structure},  as follows:
\begin{enumerate}
\item The model structure is injective, meaning all objects are fibrant.
\item The class $ex\class{F}$ of fibrant objects is the class of all exact AC-injective complexes. (See Definition~\ref{def-AC-injective complexes}.)
\item The class $\class{V}_2$ of trivial objects is $\leftperp{ex\class{F}}$.
\item The class $\class{V}_2$ of trivial objects coincides with the class of trivial objects in the exact Abs clean model structure. Therefore, their homotopy categories coincide.
\end{enumerate}
The homotopy category of the exact AC-injective model structure is equivalent to $S(\textnormal{AC-}Inj)$, the chain homotopy category of all exact AC-injective complexes.
\end{theorem}

\begin{proof}
Let $\{G_i\}$ be a generating set with each $G_i$ of finite projective dimension. Letting $C \hookrightarrow G_i$ range through all possible inclusions, set $\class{S} = \{D^n(G_i/C)\} \cup \{A_{\alpha}\} \cup \{S^n(G_i)\}$. As in the proof of Theorem~\ref{them-injective model for abs clean}, $\{A_{\alpha}\}$ is a set of complexes in $\tilclass{AC}$ which generates all others as transfinite extensions. Using~\cite[Lemma~4.5]{gillespie-degreewise-model-strucs} and Eklof's lemma, one can deduce that $\rightperp{\class{S}} \subseteq ex\class{F}$. On the other hand, since each $G_i$ has finite projective dimension we can argue that $\rightperp{\class{S}} \supseteq ex\class{F}$. Hence   $\rightperp{\class{S}} = ex\class{F}$.
So by~\cite[Corollary~2.14(2)]{saorin-stovicek} we get that $(\leftperp{ex\class{F}},ex\class{F})$ is a complete cotorsion pair. (Again it is automatically small too.) Setting $\class{V}_2 = \leftperp{ex\class{F}}$, we argue as in Theorem~\ref{them-injective model for abs clean} that we get an abelian model structure with (1), (2), and (3) all holding.

We now prove (4). Let $\class{E}$ denote the class of all exact complexes. Also, let $\class{F}$ be as in Theorem~\ref{them-injective model for abs clean} where we have already shown $(\leftperp{\tilclass{AC}}, \leftperp{\class{F}}, \dwclass{AC})$ to be a Hovey triple. Since $ex\class{F} = \class{E} \cap \class{F}$ and $\leftperp{\class{F}} \subseteq \class{E}$, Lemma~\ref{lemma-half related}(2) tells us $(\leftperp{ex\class{F}}, \class{E}, \class{F})$ is also a Hovey triple.
Using this, we compute
$$\leftperp{ex\class{F}} \cap \exclass{AC} = \leftperp{ex\class{F}} \cap (\class{E} \cap \dwclass{AC}) =  (\leftperp{ex\class{F}} \cap \class{E}) \cap \dwclass{AC} = \leftperp{\class{F}} \cap \dwclass{AC} = \tilclass{AC}.$$
So since $\tilclass{AC} = \leftperp{ex\class{F}} \cap \exclass{AC}$ and $\leftperp{\exclass{AC}} \subseteq \leftperp{ex\class{F}}$, yet another application of Lemma~\ref{lemma-half related}(2) tells us that $(\leftperp{\tilclass{AC}}, \leftperp{ex\class{F}}, \exclass{AC})$ is also a Hovey triple! Since the thick class in a Hovey triple is unique~\cite{gillespie-hovey triples}, this proves that $\class{V}_2 = \leftperp{ex\class{F}}$ is the class of trivial objects in the exact Abs clean model structure of Theorem~\ref{them-exact Abs clean model structure}.
\end{proof}

In summary, we have shown that there are three injective cotorsion pairs $$(\class{V}_1, \class{F}) \ , \ \ \ \ (\class{V}_2,ex\class{F}) \ , \ \ \ \ (\class{E}, \dgclass{I}).$$
They satisfy the hypotheses of~\cite[Theorem~4.6]{gillespie-recollements} and so immediately give the following corollary.

\begin{corollary}\label{cor-AC-injective recollement of Krause}
Let $\cat{G}$ be a locally finite dimensionally type $FP_{\infty}$ category. Then the canonical functors $S(\textnormal{AC-}Inj) \xrightarrow{} K(\textnormal{AC-}Inj) \xrightarrow{} \class{D}(\cat{G})$ induce a recollement and all three are compactly generated triangulated categories.
\end{corollary}

\section{The (exact) AC-acyclic model structures}\label{sec-AC-acyclic models}

Let $\cat{G}$ be a locally type $FP_{\infty}$ category. We now shift our focus to Gorenstein homological algebra in $\cat{G}$. Following~\cite{bravo-gillespie-hovey}, we say that a complex $I$ of injectives is \textbf{AC-acyclic} if $\Hom_{\cat{G}}(A,I)$  is an exact complex for all absolutely clean objects $A$. If $I$ itself is also exact, we call $I$ an \textbf{exact AC-acyclic} complex.
Note that in the case that $\cat{G}$ is locally noetherian, an exact AC-acyclic complex is precisely what is often called a \textbf{totally acyclic complex of injectives}. We now put a cofibrantly generated model structure on $\cha{G}$ whose homotopy category is equivalent to the chain homotopy category of all AC-acyclic complexes. In the case that $\cat{G}$ is locally finite dimensionally type $FP_{\infty}$, we put a cofibrantly generated model structure on $\cha{G}$ whose homotopy category is equivalent to the chain homotopy category of all exact AC-acyclic complexes.

\begin{theorem}\label{them-AC-acyclic model}
Let $\cat{G}$ be a locally type $FP_{\infty}$ category. Then there is a cofibrantly generated abelian model structure on $\cha{G}$ that we call the \textbf{AC-acyclic model structure}. It is an injective model structure, meaning all objects are cofibrant and the trivially fibrant objects are the injective complexes. The fibrant objects are precisely the AC-acyclic complexes of injectives. In case $\cat{G}$ is a locally finite dimensionally type $FP_{\infty}$ category, we have a similar model structure, called the \textbf{exact AC-acyclic model structure}, whose fibrant objects instead are the exact AC-acyclic complexes of injectives. The homotopy category of the AC-acyclic model structure (resp. exact AC-acyclic model structure) is equivalent to the chain homotopy category of all (resp. exact) AC-acyclic complexes of injectives.
\end{theorem}

\begin{proof}
The proof of the AC-acyclic model structure is very similar to the proof of Theorem~\ref{them-injective model for abs clean}. We let $\{G_i\}$ be a generating set. Letting $C \hookrightarrow G_i$ range through all possible inclusions, set $\class{S} = \{D^n(G_i/C)\} \cup \{S^n(A_{\alpha})\}$. Here, $\{A_{\alpha}\}$ denotes a set of absolutely clean objects in $\cat{G}$ which generates all others as transfinite extensions, using Proposition~\ref{prop-abs-clean-deconstructible}. We let $\class{F}$ denote the class of all AC-acyclic complexes of injectives. Using Eklof's lemma, one can deduce that $\rightperp{\class{S}} = \class{F}$. Therefore, by~\cite[Corollary~2.14(2)]{saorin-stovicek}, we have that $(\leftperp{\class{F}},\class{F})$ is a (small) complete cotorsion pair. Setting $\class{W} = \leftperp{\class{F}}$ we see, as in the proof Theorem~\ref{them-injective model for abs clean}, that $\class{W}$ is thick and we get the AC-acyclic model structure as described.

In the case that $\cat{G}$ is locally finite dimensionally type $FP_{\infty}$, we instead take $\class{S}' = \class{S} \cup \{S^n(G_i)\}$. Here we are assuming the $G_i$ are generators with each of finite projective dimensions. Then  $\rightperp{\class{S}'}$ consists of all the AC-acyclic complexes $F$ for which $\Ext^1_{\cha{G}}(S^n(G_i),F) = 0$. The fact that $\{G_i\}$ is a generating set implies such complexes $F$ must be exact~\cite[Lemma~4.5]{gillespie-degreewise-model-strucs}.  The fact that each $G_i$ is of finite projective dimensions implies, by a ``dimension shifting'' argument, that $\rightperp{\class{S}'}$ consists precisely of the exact AC-acyclic complexes. The existence of the exact AC-acyclic model structure, with $\class{F}' = \rightperp{\class{S}'}$ as the class of fibrant objects, now follows.
\end{proof}

\section{The Gorenstein AC-injective model
structure on $\cat{G}$}\label{sec-Gorenstein-inj}

We wish to generalize the Gorenstein AC-injective model structure from~\cite{bravo-gillespie-hovey}, to Grothendieck categories.  In this section, we do this for the locally finite dimensionally type $FP_{\infty}$ categories of Section~\ref{subsec-generators of finite projective dimension}.

\begin{setup}
Throughout this entire section we will let $\cat{G}$ be a locally finite dimensionally type $FP_{\infty}$ category [Section~\ref{subsec-generators of finite projective dimension}]. We fix a corresponding generating set $\{G_i\} \subseteq FP_{\infty}(\cat{G})$ for which each $G_i$ is of finite projective dimension.
\end{setup}

The plan is to show that the Gorenstein AC-injective cotorsion pair is cogenerated by a set, and hence by~\cite[Corollary~2.14]{saorin-stovicek} is functorially complete and small in the sense of~\cite{hovey}. It follows from~\cite[Theorem~B]{gillespie-recollements} that we automatically have a model structure on $\cat{G}$, and it is cofibrantly generated since the cotorsion pair is small. We call this the \textbf{Gorenstein AC-injective model structure on $\cat{G}$} and it is in fact Quillen equivalent to the exact AC-acyclic model structure of Theorem~\ref{them-AC-acyclic model}.

\begin{definition}
An object $M \in \cat{G}$ is called \textbf{Gorenstein AC-injective} if
$M=Z_{0}I$ for some exact AC-acyclic complex $I$ of injectives.   When $\cat{G}$ is locally noetherian this coincides with the usual notion of \textbf{Gorenstein injective}, and when $\cat{G}$ is only known to be locally coherent we call them \textbf{Ding injective}. We let $\class{GI}$ denote the class of all Gorenstein AC-injectives in $\cat{G}$ and set $\class{W}=\leftperp{\class{GI}}$.
\end{definition}

We need some
lemmas relating $\class{GI}$ and $\class{W}$ to the exact AC-acyclic model structure of Theorem~\ref{them-AC-acyclic model}. We first observe the following important fact which is necessary for obtaining completeness of the Gorenstein AC-injective cotorsion pair via cogeneration by a set.

\begin{lemma}\label{lemma-generators}
The generators $\{G_i\}$ are contained in $\class{W}$.
\end{lemma}

\begin{proof}
Let $M \in \class{GI}$. We need to show that $\Ext^1(G_i,M) = 0$. But $M = Z_0I$ for some exact AC-acyclic complex of injectives $I$. So by dimension shifting we get $\Ext^1(G_i,Z_0I) \cong \Ext^n(G_i,Z_{n-1}I)$. So our assumption that $\{G_i\}$ is a set of generators of finite projective dimension implies that this Ext group vanishes for a large enough $n$.
\end{proof}

The above lemma is in fact an instance of the following.

\begin{lemma}\label{lem-spheres}
$W \in \class{W}$ if and only
if $S^{0}(W)$ is trivial in the exact AC-acyclic model structure.
\end{lemma}

\begin{proof}
It follows from the fact that for any exact complex $X$, we have $\Ext^{1} (S^{0}W, X) = \Ext^{1} (W, Z_{0}X)$ by~\cite[Lemma~4.2]{gillespie-degreewise-model-strucs}.
\end{proof}

\begin{lemma}\label{lem-cycles-of-W}
Suppose $X$ is a complex with $H_{i}X=0$
for $i<0$ and $X_{i}$ absolutely clean for $i>0$.  Then $X$ is trivial in the
exact AC-acyclic model structure if and only if $Z_{0}X \in \class{W}$.
\end{lemma}

\begin{proof}
The proof in~\cite[Lemma~5.1]{bravo-gillespie-hovey} for $R$-modules works perfectly fine. However, we sketch an alternate proof. For this, observe that the trivial objects of the exact AC-acyclic model structure are precisely retracts of transfinite extensions of the set $\class{S}'$ in the proof of Theorem~\ref{them-AC-acyclic model}. In particular, this implies any bounded below complex of absolutely clean objects is trivial, and any bounded above exact complex is trivial. Now the given $X$ has a subcomplex $A \subseteq X$, where $A$ is the bounded below complex $\cdots \xrightarrow{} X_2 \xrightarrow{} X_1 \xrightarrow{} Z_0X \xrightarrow{} 0$. Then note that $X/A$ is (isomorphic to) the complex $0 \xrightarrow{} Z_{-1}X \xrightarrow{} X_{-1} \xrightarrow{} X_{-2} \xrightarrow{} \cdots$, which is bounded above and exact and so is trivial in the exact AC-acyclic model structure. Therefore, $X$ is trivial if and only if $A$ is trivial. But we have another subcomplex $S^0(Z_0X) \subseteq A$, whose quotient is a bounded below complex of absolutely clean objects. Thus $A$ (and hence $X$) is trivial if and only if  $S^0(Z_0X)$ is trivial. By Lemma~\ref{lem-spheres}, this happens if and only if $Z_0X \in \class{W}$.
\end{proof}

\begin{theorem}\label{thm-Gor-module}
There is an abelian model structure on $\cat{G}$, the \textbf{Gorenstein AC-injective model structure}, in which every object is cofibrant and the fibrant objects are the Gorenstein AC-injectives.
\end{theorem}

\begin{proof}
As above, we take $\class{GI}$ to be the Gorenstein AC-injective objects,
and define $\class{W}=\leftperp{\class{GI}}$.  Then
Lemma~\ref{lem-spheres} shows that $\class{W}$ is thick and contains
the injectives.  Now for any object $M$, we have a short exact sequence
\[
0 \xrightarrow{} S^{0}M \xrightarrow{} I \xrightarrow{}X \xrightarrow{} 0
\]
in which $I$ is an exact AC-acyclic complex of injectives and $X$ is
trivial in the exact AC-acyclic model category.  By the snake lemma, we get a short exact sequence
\[
0 \xrightarrow{} M \xrightarrow{} Z_{0}I \xrightarrow{} Z_{0}X \xrightarrow{} 0
\]
Of course $Z_{0}I$ is Gorenstein AC-injective by definition, but $Z_{0}X$
is in $\class{W}$ as well by Lemma~\ref{lem-cycles-of-W}, since
$X_{i}$ is injective for all $i\neq 0$ and $H_{i}X=0$ for all $i\neq
1$.  So the purported cotorsion pair $(\class{W},\class{GI})$ has
enough injectives, if it is a
cotorsion pair.

Before showing $(\class{W},\class{GI})$ is indeed a cotorsion pair we note that it will also have enough projectives. The reason is that by Lemma~\ref{lemma-generators}, we have $\{G_i\} \subseteq \class{W}$, and so the argument of Salce applies. That is, given any $M$, find a short exact sequence
\[
0 \xrightarrow{} K \xrightarrow{} G \xrightarrow{} M \xrightarrow{} 0
\]
where $G \in \class{W}$ is a direct sum of copies of objects from $\{G_i\}$. Apply the fact that we have enough injectives to get a short exact sequence
\[
0 \xrightarrow{} K \xrightarrow{} I \xrightarrow{} W \xrightarrow{} 0
\]
where $I \in \class{GI}$ and $W \in \class{W}$. Then take the obvious pushout to finish arguing $(\class{W},\class{GI})$ has enough projectives.

But we still need to show that $\class{GI} \supseteq \rightperp{\class{W}}$,
so that we know $(\class{W},\class{GI})$ is in fact a cotorsion pair.
So suppose $M \in \rightperp{\class{W}}$. We can now find a short exact
sequence
\[
0 \xrightarrow{} M \xrightarrow{} J \xrightarrow{} W \xrightarrow{} 0
\]
where $J$ is Gorenstein AC-injective and $W \in \class{W}$.  By
assumption, this must split, and so $M$ is a retract of $J$. So everything hinges on whether or not the Gorenstein injectives are closed under retracts.
But we have Lemma~\ref{lemma-retracts} below.
\end{proof}

\begin{lemma}\label{lemma-retracts}
Again let $\class{GI}$ denote the class of Gorenstein AC-injectives.
\begin{enumerate}
\item $\class{GI}$ is closed under products.
\item $\class{GI}$ is injectively resolving in the sense of~\cite[Definition~1.1]{holm}.
\item $\class{GI}$ is closed under retracts.
\end{enumerate}
\end{lemma}

\begin{proof}
For (1), note that since $(\leftperp{\exclass{I}},\exclass{I})$ is a cotorsion pair by Theorem~\ref{them-exact Inj model structure}, we have that $\exclass{I}$ is closed under direct products. It follows that $\class{GI}$ is also closed under direct products.

For (2), we note that Yang, Liu, and Liang show in~\cite[Theorem~2.7]{Ding projective} that the class of all Ding injective $R$-modules is injectively resolving in the sense of~\cite[Definition~1.1]{holm}. It means that the class contains the injectives, is closed under extensions, and is closed under taking cokernels of monomorphisms. Although the proof they give is for $R$-modules, the elegant arguments hold in the same exact way to show $\class{GI}$ is injectively resolving.

Holm shows in~\cite[Proposition~1.4]{holm} that an Eilenberg swindle argument can be used to conclude (3) from both (1) and (2).
\end{proof}

\begin{corollary}\label{prop-cogenerated}
The cotorsion pair $(\cat{W}, \cat{GI})$, where
$\cat{GI}$ is the class of Gorenstein AC-injectives, is cogenerated by
a set and $\cat{W}$ contains the generating set $\{G_i\}$. Thus the Gorenstein AC-injective model structure is
cofibrantly generated.
\end{corollary}

We use the work of {\v{S}}{\v{t}}ov{\'{\i}}{\v{c}}ek~\cite{stovicek}
on deconstructibility.  The following lemma is not stated
explicitly in~\cite{stovicek}, so we prove it here, but it is implicit
there.

\begin{lemma}\label{lem-stovicek}
Let $\class{S}$ be any set of objects containing a family of generators for $\cat{G}$. Let $(\cat{A},\cat{B})$ be the cotorsion pair cogenerated by $\class{S}$. Then there exists a set $\class{T} \subseteq
\cat{A}$ such that every element of $\cat{A}$ is a transfinite
extension of objects of $\class{T}$.
\end{lemma}

\begin{proof}
Let $\cat{D}$ be the class
of all transfinite extensions of objects of $\class{S}$.  By definition, this
is a deconstructible class in the sense of~\cite{stovicek}.  By~\cite[Corollary~2.14(2)]{saorin-stovicek}, $\class{A}$ is the class of all direct summands of $\cat{D}$. But
{\v{S}}{\v{t}}ov{\'{\i}}{\v{c}}ek~\cite[Proposition~2.9]{stovicek}
proves that this means that $\cat{A}$ is also deconstructible.  It means
there is a set $\class{T} \subseteq \cat{A}$ such that $\cat{A}$ is the class of all
transfinite extensions of $\cat{T}$.
\end{proof}

\begin{proof}[Proof of Proposition~\ref{prop-cogenerated}]
Going back to the proof of Theorem~\ref{them-AC-acyclic model}, we have an explicit set $S'$,  containing a family of generators for $\cha{G}$, for which $\rightperp{\class{S}'}$ is the class of exact AC-acyclic complexes. Thus Lemma~\ref{lem-stovicek} provides a set $\class{T}$ of complexes which generates via transfinite extensions all of the trivial objects in the exact AC-acyclic model structure.

Now let $\class{S}$ be the
collection of all objects $M$ such that $S^{0}M\in \class{T}$.  Then
$\class{T} \subseteq \cat{W}$ by Lemma~\ref{lem-spheres}. Also, if $N \in
\cat{W}$, then $S^{0}N$ is trivial in the exact AC-acyclic model, by the same lemma. So $S^0N$ must be a transfinite extension of objects of $\class{T}$.  However, each term
$X_{\alpha}$ in this transfinite extension is a subobject of
$S^{0}N$, so must be $S^{0}M_{\alpha}$ for some object $M_{\alpha}$.
It follows that $M$ is a transfinite extension of objects in $\class{S}$.
\end{proof}

\end{document}